\documentclass[aos,preprint]{imsart}
\RequirePackage[OT1]{fontenc}
\RequirePackage{amsthm,amsmath,amsfonts,amssymb}
\RequirePackage[numbers]{natbib}

\RequirePackage[colorlinks,citecolor=blue,urlcolor=blue]{hyperref}
\RequirePackage{graphicx}

\startlocaldefs

\newtheorem{theorem}{Theorem}[section]
\newtheorem{lemma}[theorem]{Lemma}
\theoremstyle{remark}

\endlocaldefs

\RequirePackage{myMATH}
\RequirePackage{graphicx, morefloats}
\RequirePackage{lineno}
\newtheorem{assumption}{Assumption}
\newtheorem{proposition}{Proposition}
\RequirePackage[colorlinks,citecolor=blue,urlcolor=blue]{hyperref}
\usepackage{float, booktabs,multirow,lscape}

\usepackage{xr}
\usepackage{xr-hyper} 
\graphicspath{{./figures/}}

\def\ns{Nowicki-Snijders} 
\newcommand{\te}[2]{$#1_{\mbox{\tiny  (#2)}}$}

\begin{document} 

\begin{frontmatter}
  
  \title{Consistent Bayesian Community Detection} 
  \runtitle{DD SBM}

  \begin{aug}
    \author{\fnms{Sheng} \snm{Jiang}\ead[label=e1]{sheng.jiang@duke.edu}}
    \and 
    \author{\fnms{Surya T.} \snm{Tokdar}\ead[label=e2]{surya.tokdar@duke.edu}}


    \affiliation{Duke University }

    \address{Department of Statistical Science, 
      Duke University,
      \printead{e1,e2}}
    
  \end{aug}

 \maketitle     

 \begin{abstract}
   Stochastic Block Models (SBMs) are a fundamental tool for community detection in network analysis. But little theoretical work exists on the statistical performance of Bayesian SBMs, especially when the community count is unknown.
   This paper studies a special class of SBMs whose community-wise connectivity probability matrix is diagonally dominant, i.e., members of the same community are more likely to connect with one another than with members from other communities.
   The diagonal dominance constraint
   is embedded within an otherwise weak prior,
   and, under mild regularity conditions, the resulting posterior
   distribution is shown to concentrate on the true community count and
   membership allocation as the network size grows to infinity. A
   reversible-jump Markov Chain Monte Carlo posterior computation
   strategy is developed by adapting the allocation sampler of
   \cite{mcdaid2013improved}. Finite sample properties are examined via
   simulation studies in which the proposed method offers competitive
   estimation accuracy relative to existing methods under a variety of
   challenging scenarios. 
  \end{abstract}


\end{frontmatter}
  
\section{Introduction}  
\label{sec:Intro} 
Community detection is the most basic yet central statistical problem in network analysis.  
To determine the number of communities, 
various tests have been constructed based on modularity \citep{zhao2011community},
random matrix theory \citep{bickel2016hypothesis, lei2016goodness}, 
and likelihood ratio \citep{wang2017likelihood}. 
Methods based on information criteria \citep{saldana2017many}
and network cross-validation \citep{chen2018network,li2020network} 
have also been designed.  
In the Bayesian realm, a stochastic block model (SBM) is often employed to
jointly infer the number of communities, 
the connectivity probability matrix, and
the membership assignment
\citep{nowicki2001estimation, mcdaid2013improved, geng2019probabilistic}. 



Despite clear empirical evidence of good statistical performance 
\cite{mcdaid2013improved, geng2019probabilistic},
theoretical guarantees on Bayesian SBMs are a rarity when the number of communities is \textit{unknown}. 
As the only exception,  
\cite{geng2019probabilistic} show 
that the community count may be consistently estimated 
under the restrictive assumptions of a homogeneous SBM with at most three communities. It is unclear if their calculations generalize to more realistic scenarios. It is also not clear if Bayesian SBMs can consistently recover the true membership allocation.


We study a special class of SBMs whose community-wise connectivity probability matrix is diagonally dominant. 
%
This special class offers a stronger encoding of the notion of {\it communities} in networks in the sense that nodes within the same community are \textit{strictly} more likely to connect with each other than with nodes from other communities. 
Crucially, the diagonal dominance condition enables membership allocations to be fully retrieved from the node-wise connectivity probabilities, as long as each community contains at least two nodes.
Of course, the node-wise connectivity probability matrix is estimated from data with statistical error. 
But as long as it is sufficiently ``close'' to the truth,
it is still possible to precisely recover the membership allocation and the community count.

For a Bayesian estimation of the diagonally-dominant SBM under a modified \ns~prior \cite{nowicki2001estimation}, we show the posterior on the node-wise connectivity matrix contracts to the truth in the sup-norm topology. Posterior contraction under sup-norm is necessary to the identification strategy detailed above. \cite{ghosh2019posterior} establish near minimax optimal posterior contraction rates in the $L_2$ norm for dense networks with the true number of communities assumed known. 
However, posterior contraction in $L_2$ or other norms that are weaker than the sup-norm do not grant the identification of the number of communities or the membership assignment from node-wise connectivity probabilities. Our sup-norm posterior contraction calculation applies the Schwartz method \citep{ghosal2000convergence, ghosal2007convergence, ghosal2017fundamentals}.
The key observation is that the sup-norm is dominated by the Hellinger distance in the special context of SBMs, so the tests required by the Schwartz method exist. 

The theoretical gains of the diagonally dominant SBMs come at the price of losing conjugacy with respect to the original \ns~prior. %
%
But posterior computation 
may be carried out with a reasonably efficient reversible-jump Markov chain Monte Carlo (MCMC) algorithm based of the allocation sampler in \cite{mcdaid2013improved}. 
%
Results from extensive numerical studies show that our Bayesian diagonally-dominant SBM 
offers comparable and competitive statistical performance against various alternatives in estimating the community count and membership assignment.

\section{The Diagonally Dominant Stochastic Block Model} 
\label{sec:main-results} 
Suppose an $n \times n$ binary adjacency matrix $A$ is observed,
with entry $A_{ij} = 1$ if node $i$ and node $j$ are connected
and $A_{ij} = 0$ otherwise.  
The stochastic block model (SBM) assumes there are 
$K \in \Z_+$ communities among the $n$ nodes and the connection
between nodes \textit{exclusively} depends on their community membership. 
The community assignment $Z$ partitions nodes $ \{1,...,n\}$ into $K$
non-empty groups and assigns each node with a community label.  
Let the community-wise connectivity probability matrix be $P \in [0,1]^{K \times K}$.
Then, 
\begin{equation}
  \label{eq:SBM} 
 A_{ij}|Z \distras{ind} 
  Ber({P_{{Z(i)}{Z(j)}}})
  \text{ for } 1 \le i < j \le n,
\end{equation}
and $P\left( {{A_{ii}} = 0|Z} \right) = 1$
for $i\in \{1,...,n\}$, assuming no self-loops.
We denote the above SBM model as $SBM(Z,P,n,k)$. 
Due to its simplicity and expressiveness, 
SBM and its variants are fundamental tools for community detection   
\citep[e.g.,][]{karrer2011stochastic, airoldi2008mixed, peixoto2014hierarchical}.

\subsection{Bayesian SBM with conjugate priors} 
\label{sec:stochastic-blockmodel}

For Bayesian estimation of the SBM, \cite{nowicki2001estimation} propose the following conjugate prior: given $K$, 
\begin{equation}
  \label{eq:NS-prior} 
 \begin{split}
{P_{ab}} &\distras{iid} U\left( {0,1} \right), a,b = 1,...,K \\
Z_i &\distras{iid} MN(\pi),  i= 1,...,n \\
\pi &\sim Dir\left( \alpha  \right). 
\end{split}
\end{equation} 
This prior is widely used and adapted to more complicated cases 
in the Bayesian SBM literature
\citep{ghosh2019posterior, van2018bayesian,geng2019probabilistic,mcdaid2013improved}.

For the unknown $K$ case,
to maintain conjugacy, it is natural to place a Poisson prior on $K$ 
\citep{mcdaid2013improved, geng2019probabilistic}.
With conjugacy, \cite{mcdaid2013improved} marginalize out $P$ from the posterior
$\Pi_n(Z,K,P|A)$ 
and develop an efficient ``allocation sampler'' to directly sample from $\Pi_n(Z,K|A)$;  
\cite{geng2019probabilistic} adapt the idea of MFM sampler of \cite{miller2018mixture}
to the SBM case:   
marginalize out $K$ from the posterior $\Pi_n(Z,K,P|A)$, 
and develop a Gibbs sampler sampling from $\Pi_n(Z,P|A)$. 

 \subsection{Our proposal: diagonally dominant SBM}
 \label{sec:proposal}

In this paper, we propose to modify the conjugate specification of
Nowicki and Snijders' prior on the connectivity matrix
$P$ by imposing a diagonal dominance constraint. 
The constraint is imposed in two steps: 
first specify a prior distribution for the diagonal entries of $P$, 
then conditional on the diagonal entries, specify a prior distribution on
the off-diagonal entries such that the off-diagonal entries are
strictly less than their corresponding diagonal entries.

For instance, we specify the following prior:   
\begin{equation}
  \label{eq:proposed-prior} 
  \begin{split}
   & {P_{aa}}|K, \delta
    \distras{iid} U(\delta,1], a\in \{1,...,K\},  \\ 
   & {P_{ab}}|K, \delta, \{P_{aa}\}_{a\in \{1,...,K\}}
     \distras{ind} U(0,P_{aa}\wedge P_{bb} -\delta),
    a<b \in \{1,...,K\},     \\
    &\delta  \propto \log(n)/n, \\ 
    & K \sim Pois(1), 
  \end{split}
\end{equation}
where the hyperparameter $\delta$ is chosen to be a deterministic sequence
that goes to 0 as the network size grows to infinity.
Uniform distributions in (\ref{eq:proposed-prior}) are used for simplicity
and can be replaced with other distributions.

In contrast to the Nowicki and Snijders' priors, our prior specification directly imposes conditional dependence between diagonal entries and off-diagonal entries. 
The dependence matches the idea of ``community'' at the price of losing conjugacy.

The modification is mainly for two reasons.
Firstly, the prior constraint of diagonal dominance offers a neat identification of
the number of communities, and  
allows us to consistently estimate the number of communities and
membership.
(See more details in section \ref{sec:ident-via-diag}.)

Secondly, the resulting posterior under the modified prior is more interpretable.  
Though the prior specification following \cite{nowicki2001estimation} is conjugate, 
off-diagonal entries can be greater than diagonal entries under the prior, that is,
nodes can be more likely to be connected to nodes from other communities than nodes from their own community. 
Such configurations violate the idea of ``community''. 
Consequently, posterior samples of connectivity matrices can violate diagonal dominance and are hard to interprete within the framework of SBM.  

\subsection{$L_2$ minimax rate}
\label{sec:minimax-rates}

This paper studies a special sub-class of SBM. 
One may wonder if the diagonally dominant (DD) SBM actually solves a simpler 
community detection problem.
To answer this question, we calculate the $L_2$ minimax rate of estimation for
DD-SBM and compare it with the minimax rates derived in \cite{gao2015rate}.   

Now, we define the parameter space of DD-SBM.
DD-SBM has the following space of connectivity matrix 
\begin{equation}
  \label{eq:connectivity-matrix-space}
  {S_{k,\delta}} =
  \left\{ {P \in {{\left[ {0,1} \right]}^{k \times k}}:
    {P^T} = P,
    {P_{ii}} > \delta +
    \mathop {\max }\limits_{j \ne i}
    \left( {{P_{ij}}} \right),
    i \in \{1,...,k\} } \right\},
\end{equation}
where $\delta \in [0,1)$ is a constant.  
The key departure from the literature is the diagonal dominance constraint:
$  {P_{ii}} > \delta + \mathop {\max }\limits_{j \ne i} \left( {{P_{ij}}} \right)$,
for all $i \in \{1,...,k\}$. 
Under this constraint, between community connection probabilities are
less than within community connection probabilities by $\delta$. 
The gap is inherited by the node-wise connectivity probability matrix.  

Further with the membership assignment $Z$, 
we can define the space for node-wise connectivity probability matrix:
\begin{equation}
  \label{eq:node-connectivity-space-dd}
  {\Theta _{k,\delta}} = \left\{ {
      T({ZPZ^T})  \in [0,1]^{n \times n}:
      P \in S_{k,\delta},
      Z \in {\cZ_{n,k}}} \right\},
\end{equation}
where  
$\cZ_{n,k}$ denotes the collection of all possible assignment
of $n$ nodes into $k$ communities which have \textit{at least two} elements,
and $T(M) := M - \diag(M)$ for any square matrix $M$.   
The node-wise connectivity probability matrix inherits the structural
assumption of diagonal dominance.
The minimum community size assumption allows
recovering community membership from node-wise connectivity probability matrix. 
It is worthwhile to emphasize that singleton communities are ruled
out. 

The following $L_2$ minimax result implies that DD-SBM estimation is as difficult as the original SBM estimation problem, 
as long as the dominance gap is shrinking at certain rate. 
In our calculation, the gap squared ($\delta^2$) is dominated by
the ``clustering rate''  $\log(k)/n$ 
\citep{gao2015rate, gao2018minimax, klopp2017oracle}.

\begin{proposition}
  \label{prop:minimax}
  For any $ k \in \{1,..,n\}$ and $\delta \precsim  \sqrt{\log(k)/n}$,  
  \begin{equation}
    \label{eq:minimax-dd-l2}
    \inf \limits_{\hat{\theta}}
    \sup \limits_{\theta \in \Theta_{k,\delta}}
    \bE \left[||\hat{\theta}-\theta||_{2}^2\right]
    \asymp 
      \frac{k^2}{n^2}+\frac{\log(k)}{n}.  
  \end{equation}
\end{proposition}

\begin{proof}
  The upper bound follows theorem 2.1 of \cite{gao2015rate} as
  the diagonally dominant connectivity matrix space is a subset of the unconstrained
  connectivity matrix space.

  The lower bound follows the proof of theorem 2.2 of \cite{gao2015rate} but their 
  construction violates the diagonally dominant constraint.
  It turns out a diagonally dominant version of their construction is available. 
  For brevity, we only highlight the differences from the proof in \cite{gao2015rate}.
  
  For the nonparametric rate, we construct the $Q^\omega$ matrix  by
  $Q^\omega_{ab} = Q^\omega_{ba}
  = \frac{1}{2} - \delta - \frac{c_1k}{n} \omega_{ab}$, for $a>b \in \{1,...,k\}$
  and $Q^\omega_{aa} = \frac{1}{2}$, for $a \in \{1,...,k\}$.
  The rest of the proof for the nonparametric rate remains the same. 

  For the clustering rate, we construct the $Q$ matrix with the following form
 $ \begin{bmatrix}
 D_1  & B \\
B^T & D_2 
\end{bmatrix}$, 
where $D_1 = \frac{1}{2} I_{k/2}$,
$B$ follows the same construction of \cite{gao2015rate} except that
$B_a = \frac{1}{2} -\delta -  \sqrt{\frac{c_2 \log k}{n}} \omega_a$ for $a \in \{1,...,k/2\}$,
$D_2 = (\frac{1}{2} - \delta -  \sqrt{\frac{ \log k}{n}} )1_{k/2}1^T_{k/2}
+  (\delta + \sqrt{\frac{ \log k}{n}}) I_{k/2}$.
As $\delta \precsim  \sqrt{\log(k)/n}$, the KL divergence upper bound remains the same. 
The rest of the proof for the clustering rate remains the same
as the entropy calculation and the volume argument are unaffected.

\end{proof}

\section{Consistent Bayesian Community Detection} 
\label{sec:cons-est-k}  

\subsection{Identification Strategy}
\label{sec:ident-via-diag} 

The first consequence of diagonal dominance is that
the node-wise connectivity probability matrix spaces of different
ranks are non-overlapping. 
This observation offers a neat partition of the parameter
space by the number of communities. 
\begin{lemma}
  \label{lemma:kid}
  Suppose $k \neq k^\prime \in \N$, then 
  $\Theta_{k,\delta} \cap \Theta_{k^\prime,\delta^\prime} = \emptyset$  
  for any $\delta, \delta^\prime \ge 0 $.  
\end{lemma}

Secondly, with diagonal dominance, 
it is possible to exactly identify the number of communities, the
membership of every node and
the community-wise connectivity probability matrix from 
node-wise connectivity probability matrix under mild conditions. 
A more rigorous statement is presented in Lemma \ref{lemma:T-inverse}.
The recovery is based on checking each node's connectivity probabilities with
 other nodes, as each node is connected with nodes from its own community with the
highest probability.  
\begin{lemma}
  \label{lemma:T-inverse}
  Suppose $P \in S_{k,\delta}$ for some constant $\delta>0$,
  $\theta = T(ZPZ^T)  $ for some $Z \in \cZ_{n,k}$, 
  $T^{-1}$ recovers both community assignment $Z$ and connectivity matrix $P$ from
  $\theta$.  
\end{lemma}

\begin{proof}
  Without loss of generality, assume the nodes are ordered by community and
  we can write $Z=[ {\bf{1}}_{n_{1}},...,{\bf{1}}_{n_k}]$ where
  $n_j$ denotes the number of nodes in community $j$ and
  ${\bf{1}}_{n_j}$ is a $n \times 1$ vector with entries in the $j^{th}$
  block being 1. 
  Therefore, the off-diagonal terms of $\theta$ are the off-diagonal terms of
  $  ZP{Z^T} $. 

  Suppose we hope to pin down $i^{th}$ node's community membership.
  We take $i^{th}$ row of $\theta$ and it contains the connectivity
  probabilities of node $i$ and all other nodes. 
  As $Z \in \cZ_{n,k}$ whose minimum community size is two,  
  $\cC_i \equiv \{j: \theta_{ij} = \mathop {\max }\limits_{\ell}\theta_{i\ell}\}$
  is \textit{exactly} the set of node(s) from the community of node
  $i$.
  If $\cC_i$ contains node(s) from other communities, then
  the connectivity probabilities of node $i$
  with those node(s) are cross-community  which are strictly less than the
  within-community connectity probabiilty of node $i$,
  contradicting the construction of $\cC_i$.
  If $\cC_i$ misses node(s) from the community of node $i$,
  then the connectivity probabilities of node $i$
  with those node(s) are within-community  which have
  to match the connectivity probabilities of nodes in $\cC_i$.  
  Therefore, by enumerating the above procedure for all rows of
  $\theta$, $Z$ is identified up to a permutation of columns.  

  To recover $P$ from $\theta$, it suffices to use $Z$ 
  and plug in corresponding values from $\theta$. 
  
\end{proof}

In practice, the \textit{exact} knowledge of node-wise connectivity probability matrix is not
available. 
However, the precise recovery in Lemma \ref{lemma:T-inverse} is possible with 
the estimated node-wise connectivity probability matrix.
This is formalized in Lemma \ref{lemma:ID-Z}. 
We use sup-norm to characterize the accuracy of the knowledge of node-wise connectivity probability matrix.
For any node-wise connectivity matrix $\theta^0$,
there exists $Z_0$ and $P^0$ such that $\theta^0 = T(Z_0P^0 Z_0^T)$.
Without loss of generality, we can fix the column ordering of $Z_0$ so that
$P^0$ is consequently defined.    

\begin{lemma}
  \label{lemma:ID-Z}
  Suppose ${\theta ^0} = T(Z_0P^0 Z_0^T)$
  for some $Z_0 \in \cZ_{n,{k_0}}$, ${P^0} \in {S_{{k_0},\delta }}$ and $\delta >0$. 
  Then,
  $ \{\theta = T(ZPZ^T): ||\theta-\theta^0||_{\infty} \le r,
  Z\in \cZ_{n,k}, P \in S_{k,\delta} \} =
  \{  T(Z_0PZ_0^T): ||P-P^0||_{\infty} \le r, P \in {S_{k_0, \delta}} \} $
  holds for all $r < \delta /2$. 
\end{lemma}

\begin{proof}
  

  Pick any $\theta  \in
  \{\theta = T(ZPZ^T): ||\theta-\theta^0||_{\infty} \le r,
  Z\in \cZ_{n,k}, P \in S_{k,\delta} \} $, 
  define $\cC_i=\{j: \theta_{ij}
  = \mathop {\max}\limits_{\ell}\theta_{i\ell}\}$; 
similarly, for $\theta^0$, define 
    $\cC_i^0=\{j: \theta^0_{ij} = \mathop {\max }\limits_{\ell}\theta^0_{i\ell}\}$. 
  The statement is equivalent to
  $\cC_i = \cC_i^0$ for all $i \in \{1,...,n\}$ and all $\theta $.

  First, note for any $j \in \cC^0_i$ and  $\ell \in \{1,...,n\}\backslash \cC^0_i$,   
  ${\theta _{ij}} - {\theta _{i\ell}} 
  = {\theta _{ij}} - \theta _{ij}^0 
  + \theta _{ij}^0 - \theta _{i\ell}^0 
  + \theta _{i\ell}^0 - {\theta _{i\ell}} 
  > \delta  - 2r > 0$.
  That is, $\cC^0_i$ identifies a set of nodes with higher connectivity probabilities with node $i$ relative to nodes from $\{1,...,n\}\backslash \cC^0_i$.
  Recall $\cC_i$ is the collection of nodes with the highest connectivity probability. 
  Then,  $\cC_i \subseteq \cC_i^0 $ for all $i \in \{1,...,n\}$. 

  If $\cC_i^0$ contains nodes from at least two communities of $\theta$,
  then there exist $j_1, j_2 \in \cC^0_i$, such that
  $|\theta_{i j_1} - \theta_{i j_2}| > \delta$
  as $P\in S_{k,\delta}$. 
  Note  for all $j_1, j_2 \in \cC^0_i$, $\theta _{ij_1}^0 = \theta _{ij_2}^0$, then it follows 
    $|{\theta _{ij_1}} - {\theta _{ij_2}}|  
  = |{\theta _{ij_1}} - \theta _{ij_1}^0 
  + \theta _{ij_1}^0 - \theta _{ij_2}^0 
  + \theta _{ij_2}^0 - {\theta _{ij_2}} |
  \le  |{\theta _{ij_1}} - \theta _{ij_1}^0| + |\theta _{ij_2}^0 - {\theta _{ij_2}} |
  \le 2r <\delta $. 
  Then, the contradiction implies $\cC_i = \cC_i^0$ for all $i$.
  As $\theta $  is  arbitrary, 
    $\cC_{i}=\cC_{i}^0$ for all $i \in \{1,...,n\}$ and for all $\theta$.  
\end{proof}

\subsection{Posterior Concentration} 
\label{sec:post-contr}

To study the asymptotic behavior of the diagonally dominated SBM,  
we make the following assumptions on the prior specification.
The prior specification in Assumption \ref{assumption:DD} and
\ref{assumption:prior} is indexed by  $n$, the number of nodes in the
network , and can be interpreted as a sequence of prior distributions.

\begin{assumption}(Prior mass on the parameter space)
  \label{assumption:DD}
  There exists $\bar{\delta}\in (0,1)$ such that
  for all $0< \delta < \bar{\delta}$ and $k>1$,  
  ${\Pi _n}\left( {{S_{k,\delta }}}|K=k \right)
  \ge 1 - {e^{ - {n^2}{\delta }}}$. 
\end{assumption}

Assumption \ref{assumption:DD} requires that the prior specification 
is essentially diagonally dominant. 
Under Nowicki and Snijders' prior, conditional on $k$ communities,
the prior probability of diagonal dominance is $1/k^k$. 
Therefore,  Nowicki and Snijders' prior does not satisfy Assumption \ref{assumption:DD}. 

\begin{assumption}(Prior decay rates) 
  \label{assumption:prior}
  \begin{enumerate} 
  \item (Prior on $P$ conditional on $K$ and $\delta$)

    For $a \in \{1,...,k\}$, diagonal entries $\{P_{aa}\} $ are independent with prior density  
    $\pi_n (P_{aa}|K,\delta)
    \ge
    e^{-C \log(n) P_{aa}} 1_{ \{P_{aa}\in (\delta, 1)\}}$ 
    for some positive constant $C$ independent of $a\in \{1,...,k\}$.

    For $ a<b \in \{1,...,k\}$, off-diagonal entries $\{P_{ab}\}_{a\in \{1,...,k\}}$ are conditionally independent 
    on diagonal entries with conditional prior density
    \begin{equation}
      \label{eq:prior-off-diag}
      {\pi _n}\left( {P_{ab}} | \{P_{aa}\}_{a\in \{1,...,k\}},  \delta, K \right)
    \ge e^{ - C  \log(n) ({P_{aa}} \wedge {P_{bb}}) } 1_{\{ P_{ab} \in [0, {P_{aa}} \wedge {P_{bb}}-\delta]\}}
    \end{equation} 
    for  
    some positive constant $C$ independent of $a,b\in \{1,...,k\}$.  

  \item (Prior on $Z$ conditional on $K$)

    The prior on the membership assignment $Z$ satisfies 
    ${\Pi _n}\left( {Z = z|K = k} \right) \ge {e^{ - C n\log (k)}}$
    for all $z \in \cZ_{n,k}$ 
    and 
    for some universal positive constant $C$. 
\item (Prior on $K$)

  The support of $K$ is $[K_n]$ with $K_n \precsim \sqrt{n}$. 
  For $k \in [K_n]$, the prior on $K$ satisfies  
  ${\Pi _n}\left( {K = k} \right) \ge e^{- C k\log (k)}$ 
  for some universal positive constant $C$. 
  \end{enumerate}

\end{assumption}

Assumption \ref{assumption:prior} makes more specific decay rate assumptions on the prior mass of connectivity matrix $P$, the assignment $Z$, and the number of
communities $K$. 
The rate assumption of the prior on $P$ given $K$ and $\delta$
essentially requires the prior density on $P$ is lower bounded away from 0. 
For instance, the uniform prior on $P$ and the Poisson prior on $K$
in (\ref{eq:proposed-prior}) 
 satisfy Assumption \ref{assumption:prior}.

\begin{theorem}
  \label{thm:posterior-contraction}
  Suppose adjacency matrix $A \sim SBM (Z_0, P^0, n, k_0)$,
  let $\theta ^0 = T(Z_0P^0Z_0^T)$, 
  $P^0 \in \Theta_{k_0,\delta_0}$ 
  for some $k_0 \precsim \sqrt{n} $
  and $\delta_0 > 0$, and the number of zero and one entries of $\theta^0$ is
  at most $O(n^2\varepsilon_n)$ where 
  $\varepsilon _n^2 \asymp \frac{\log(k_0)}{n}$.   
  The prior $\Pi_n$ satisfies Assumption \ref{assumption:DD} and \ref{assumption:prior}.
  Then, for all sufficiently large $M$,   
  \[{\bP_{0,n}}{\Pi _n}
    \left( {\theta: 
        ||\theta  - \theta ^0||_\infty  \ge 
        M \varepsilon _n} | A \right) \to 0.\]
\end{theorem}

The proof of Theorem \ref{thm:posterior-contraction} follows Schwartz method
\citep{schwartz1965bayes, barron1999consistency,ghosal1999posterior,ghosal2007convergence}. 
Details of the proof are deferred to Section \ref{sec:proof-theorem}.

Though exact  $L_\infty$ minimax rates of SBM or DD SBM are unknown,  
$L_\infty$ minimax rates are lower bounded by $L_2$ minimax rates. 
The $L_2$ minimax rate calculation of DD SBM in Proposition \ref{prop:minimax} can be useful
for judging the sharpness of the posterior contraction rate in Theorem \ref{thm:posterior-contraction}. 
As we assume $k_0 \precsim \sqrt{n}$,
the posterior contraction rate in $||\cdot||_\infty$  matches
the $L_2$ minimax rates in Proposition \ref{prop:minimax}, 
and the posterior contraction rate is minimax-optimal.


With Theorem \ref{thm:posterior-contraction} and Lemma \ref{lemma:ID-Z},
we can establish the consistent estimation of the true number of communties and
true membership assignment.
The main result is summarized as follows. 
\begin{theorem}
  \label{thm:main}
  Under the same assumptions of Theorem \ref{thm:posterior-contraction},    
  \[ {\bP_{0,n}}\left[ {{\Pi _n} \left(
          \{K=k_0\} \cap \{Z = {Z_0}\} |A \right)} \right] \to 1. \] 

\end{theorem}

\begin{proof}
  In light of Theorem \ref{thm:posterior-contraction}, the posterior mass is essentially
on $\{\theta: ||\theta  - \theta _n^0|{|_\infty } \le {\varepsilon _n}\}$. 
Therefore, we leverage Lemma \ref{lemma:ID-Z} to identify $k_0$ and $Z_0$ on the set.

  Define $E_0= \{K=k_0\} \cap \{Z=Z_0\}$. 
  Note the decomposition
  \[ E^c_0
  = \left( E^c_0 \cap \{||\theta  - \theta ^0|{|_\infty } \le {\varepsilon _n}  \}\right)
  \cup 
  \left( E^c_0 \cap \{||\theta  - \theta ^0|{|_\infty }> {\varepsilon _n}  \}\right) \]  
for some $\varepsilon_n$, then
  \begin{equation}
    \label{eq:main}
    {\Pi _n}\left( {E^c_0|A} \right) \le {\Pi _n}\left( {E^c_0,
        ||\theta  - \theta ^0|{|_\infty } \le {\varepsilon _n}|A} \right)
    + {\Pi _n}\left( {||\theta  - \theta ^0|{|_\infty } > {\varepsilon _n}|A} \right)    
  \end{equation}
  where $\varepsilon_n \to 0$ is chosen to match the posterior contraction rate in sup-norm. 

  Then,  the posterior probability of choosing wrong number of communities
  or wrong membership assignment can be upper bounded
  via the identification assumption and 
  convergence of the posterior distribution of $\theta$. 
  For the first part of Equation (\ref{eq:main}),
  the $\delta$ gap assumption of $\theta^0$ satisfies
  $\delta_0 \succsim \varepsilon_n$. Then, by Lemma \ref{lemma:ID-Z},
  for all sufficiently small $\varepsilon_n$,
  $\{||\theta  - \theta ^0|{|_\infty } \le {\varepsilon _n}\}$ is the same as its
  $Z_0$ slice where the implied number of communities is $k_0$. 
  
  For the second part, Theorem \ref{thm:posterior-contraction} implies
  $\bP_0[{\Pi _n}\left( {||\theta  - \theta ^0|{|_\infty } > {\varepsilon _n}|A} \right) ]\to 0$. 

\end{proof}

\subsection{Proof of Theorem \ref{thm:posterior-contraction}}
\label{sec:proof-theorem}

Pioneered by \cite{schwartz1965bayes} and further developed by
\cite{barron1999consistency,ghosal1999posterior,ghosal2007convergence}, 
Schwartz method is the major tool to study posterior concentration
properties of Bayesian procedures as sample size grows to infinity
\citep{ghosal2017fundamentals}. 
Schwartz method seeks for two sufficient conditions to guarantee posterior concentration:
the existence of certain tests and 
prior mass condition. 
The existence of certain tests often reduces to the construction of certain sieves and an entropy condition associated with the sieve,
if the metric under which we wish to obtain posterior contraction is dominated by Hellinger distance.   
The prior mass condition requires sufficient amount of prior mass on some KL neighborhood near the truth.

Establishing convergence in $||\cdot||_\infty$ via the general framework of Schwartz method requires 
$||\cdot||_\infty$ to be dominated by Hellinger distance.
In general, $||\cdot||_\infty$ is (weakly) stronger than Hellinger distance and not
dominated by Hellinger distance. 
However, in the special case of SBM, the parameter space is constrained and the desired dominance holds.
This observation is shown in Lemma \ref{lemma:norm-dominance}. 

\begin{lemma}
  \label{lemma:norm-dominance}
  Suppose $A_{ij} | \theta \distras{IND} Ber(\theta_{ij}) $ for $i<j $ and $i,j \in \{1,...,n\}$,
  then $||\cdot||_{\infty}$ is dominated by Hellinger distance:  
  $||{\theta ^0} - {\theta ^1}|{|_\infty } \le
  2 H\left( {\bP_{\theta ^0}, \bP_{\theta ^1}} \right)$.
\end{lemma}

With the norm dominance, the existence of certain tests reduces to construct
a suitable sieve which charges sufficient prior mass
and whose metric entropy is under control. 
In our proof, the sieve is constructed as the
set of all well separated node-wise connectivity probability matrices: 
$\bigcup \nolimits_{k = 1}^{{K_n}} {{\Theta _{k,\delta_n}}}$ 
for some carefully chosen $\delta_n$ and $K_n$.

In light of Lemma \ref{lemma:kid}, the metric entropy of the sieve can be
neatly bounded. 
The entropy calculation is summarized in Lemma \ref{lemma:entropy}. 
\begin{lemma}
  \label{lemma:entropy} 
  Suppose $\varepsilon _n \to 0$ as $n \to \infty$,
 and  $\varepsilon _n \precsim \delta_n $, then 
  metric entropy satisfies 
  \begin{equation}
    \label{eq:entropy} 
    \log N\left( {{\varepsilon _n},
        \bigcup \nolimits_{k = 1}^{{K_n}} {{\Theta _{k,\delta_n}}} ,
        || \cdot |{|_\infty }} \right)
    \precsim
    \left( {n + 1} \right)\log {K_n}
    + \frac{1}{2}{K_n}\left( {{K_n} + 1} \right)
    \log \left( {1/{\varepsilon _n}} \right). 
  \end{equation}
\end{lemma}

The prior mass condition in terms of KL divergence can be reduced 
to a prior mass condition in terms of  $||\cdot||_{\infty}$ norm.  
This observation is summarized in Lemma \ref{lemma:KL-dominance}.

\begin{lemma}
  \label{lemma:KL-dominance}
  The observation model is $A_{ij} | \theta^0 \distras{IND}
  Ber(\theta^0_{ij}) $  
  for $i<j $ and $i,j \in \{1,...,n\}$. 
  Suppose  $C_0 ={{\mathop {\min }\nolimits_{i < j:0 < \theta _{ij}^0 < 1}
      \theta _{ij}^0\left( {1 - \theta _{ij}^0} \right)}}>0 $, 
  and the number of zero and one entries of ${\theta ^0}$ is less than
  $O(n^{2}\varepsilon_n)$ for some  $\varepsilon_n \to 0 $ such that 
  $n^2 \varepsilon_n \to \infty$. 
 If $||{\theta} -  \theta^0 |{|_\infty } \le \varepsilon_n $, 
  then $KL\left( {{\bP_{{\theta ^0}}},{\bP_\theta }} \right)
  \precsim
  C_0^{-1} {n^2} \varepsilon_n^2 $,  
  and ${V_{2,0}}\left( {{\bP_{{\theta ^0}}},{\bP_\theta }} \right)
  \precsim
  C_0^{-1}{n^2} \varepsilon_n^2 $. 
\end{lemma}

Lemma \ref{lemma:KL-dominance} simplifies the prior mass condition to element-wise
probability calculation.  
Immediately with Assumption \ref{assumption:prior}, we obtain the following
prior mass calculation. 
\begin{lemma}[prior mass condition]
  \label{lemma:prior-mass}
  Suppose $P^0 \in S_{k_0,\delta_0}$ for some $k_0 \precsim \sqrt{n}$  
  and constant $\delta_0\in (0,1)$,  
  and  $\varepsilon_n^2 \asymp \log(k_0)/n$, 
  then 
  under Assumption \ref{assumption:prior},
  there exists a constant $C$ only dependent on $P^0$ and $C_0$ such that 
  \begin{equation}
    \label{eq:prior-mass}
   {\Pi _n}\left( { P:
       ||P - {P^0}|{|_\infty } < C_0{\varepsilon_n} };
    Z=Z_0;K=k_0 |\delta \right)
  \ge e^{-C n^2 \varepsilon_n^2} 
  \end{equation}
  holds for all sufficiently large $n$. 
\end{lemma}

With the above preparation, the proof of Theorem \ref{thm:posterior-contraction} is as follows. 
The structure of the proof follows \cite{ghosal2007convergence}. 

\begin{proof} 
  We first verify prior mass condition. 
  By Lemma \ref{lemma:KL-dominance}, the set 
\[\left\{ {\theta  \in \bigcup \nolimits_{k = 1}^{{K_n}} {{\Theta _{k,0}}} :
      KL\left( {{\bP_{{\theta_n ^0}}},{\bP_\theta }} \right) < n^2\varepsilon _n^2,
      {V_{2,0}}\left( {{\bP_{{\theta_n ^0}}},{\bP_\theta }} \right) <
      n^2 \varepsilon _n^2} \right\}\] 
  contains a sup-norm ball 
$\left\{ {\theta  \in \bigcup \nolimits_{k = 1}^{{K_n}} {{\Theta _{k,0}}}: 
    ||\theta  - {\theta_n ^0}|{|_\infty } < C_0{\varepsilon _n}
    } \right\}$
for some constant $C_0$ only dependent on $\theta^0$.  
Choose $1 \succ {\tau _n} \succ {\varepsilon _n} $,
the sup-norm ball further contains the following sup-norm ball 
$\left\{ {\theta  \in {{\Theta _{k_0,\tau_n}}}: 
    ||\theta  - {\theta_n ^0}|{|_\infty } < C_0{\varepsilon _n}
    } \right\}$. 
By Lemma \ref{lemma:ID-Z}, the sup-norm ball is essentially its $Z_0$
slice which reduces to 
\[{\Pi _n}\left( { P\in S_{k_0,\tau_n} :
        ||P - {P^0}|{|_\infty } < C{\varepsilon _n}
          }; Z=Z_0;K=k_0 \right). \]  
By Lemma \ref{lemma:prior-mass}, the prior mass is further lower
bounded by $e^{-Cn^2\varepsilon_n^2}$ for some constant $C$ only
dependent on $P^0$ and $C_0$. 



  Next, we check the existence of tests. 
  The existence of tests boils down to metric entropy condition and
  prior mass condition of the sieve. 
  The sieve is constructed as $\bigcup \nolimits_{k = 1}^{{K_n}} {{\Theta _{k,\delta_n}}}$
  with $1 \succ {\delta _n} \succsim {\varepsilon _n^2}$.

  Metric entropy condition of the sieve requires the metric entropy is upper bounded by
  $C n^2\varepsilon_n^2$. 
  Clearly, this is satisfied by Lemma \ref{lemma:entropy}.

  It is left to show the prior mass on the sieve.
  Note ${\Pi _n}\left( {{{\left( {\bigcup\nolimits_{k = 1}^{{K_n}}
              {\Theta _{k,{\delta _n}}^{}} } \right)}^c}} \right)
  \le {\Pi _n}\left( {\Theta _{K_n,{\delta _n}}^c} \right)
  = {\Pi _n}\left( {\Theta _{K_n,{\delta _n}}^c|K = K_n} \right)
  \Pi_n(K=K_n)$,  
  then the prior mass on the sieve is also satisfied by
  a union bound: 
  \[\begin{array}{lll}
      {\Pi _n}\left( {\Theta _{k,{\delta _n}}^c|K = k} \right)
      &\le&
            \sum\nolimits_{z \in {\cZ_{n,k}}}^{}
            {{\Pi _n}\left( {\Theta _{k,{\delta _n}}^c|Z = z,K = k} \right)
            {\Pi _n}\left( {Z = z|K = k} \right)} \\
      &\le&
            \mathop {\max }\nolimits_{z \in {\cZ_{n,k}}}
            {\Pi _n}\left({\Theta _{k,{\delta _n}}^c|Z = z,K = k} \right) \\
      &=&             \mathop {\max }\nolimits_{z \in {\cZ_{n,k}}}
            {\Pi _n}\left({T(zPz^T):P \in S _{k,{\delta _n}}^c|Z = z,K = k} \right) \\
      &\le& {\Pi _n}\left({S _{k,{\delta _n}}^c|K = k} \right)  \\
      &\le& e^{- n^2 \delta_n}\\ 
      &\precsim& {e^{ - C{n^2}\varepsilon _n^2}}
    \end{array}\]
  for some constant $C$.

\end{proof}

\section{Posterior Sampler and Inference} 
\label{sec:sampler}

\subsection{Reversible-jump MCMC algorithm}
\label{sec:AS-MH}

Under the diagonally dominant prior (\ref{eq:proposed-prior}),
the posterior distribution is as follows,
\begin{equation}
  \label{eq:posterior}
  {\Pi _n}\left( {Z,K,P|A} \right)
  \propto
  \Pi \left( {A|Z,P} \right)
  {\Pi _n}\left( {P|Z} \right)
  {\Pi _n}\left( {Z|K} \right)
  {\Pi _n}\left( K \right)
\end{equation}
with
\[\begin{array}{lll}
    \Pi \left( {A|Z,P} \right)
    &=&
        \prod\nolimits_{1 \le a \le b \le K}
        {P_{ab}^{{O_{ab}}\left( Z \right)}
        {{\left( {1 - {P_{ab}}} \right)}^{{n_{ab}}\left( Z \right) - {O_{ab}}\left( Z \right)}}} \\
    {\Pi _n}\left( {P|Z,K,\delta_n} \right)
    &=&
        \prod\nolimits_{1 \le a < b \le K}
        {\frac{{{1_{\left( {0 \le {P_{ab}}
        \le
        \left( {{P_{aa}} \wedge {P_{bb}}} \right) - {\delta _n}} \right)}}}}
        {{\left( {{P_{aa}} \wedge {P_{bb}}} \right)
        - {\delta _n}}}} \\
    {\Pi _n}\left( {Z|K} \right)
    &=&
        \frac{{\Gamma \left( K \right)}}
        {{\Gamma \left( n+K \right)}}
        \prod\nolimits_{1 \le c \le K}
        {\Gamma \left( {{n_c}\left( Z \right)}+1 \right)} \\
    {\Pi _n}\left( K \right)
    &\propto&
              \frac{1}{{K!}}{1_{1 \le K \le {K_n}}}. 
\end{array}\]

For comparison, the Nowicki and Snijders' prior is conjugate and the
community-wise connectivity probability matrix
$P$ can be marginalized out in the posterior distribution.
Therefore, posterior inference on $K$ is directly based on posterior
draws from $\Pi_n(Z,K|A)$. 
However, the truncated Nowicki and Snijders' prior loses conjugacy.
Our posterior inference needs to sample from $\Pi_n(P,Z,K|A)$.  

We propose an Metropolis-Hastings algorithm to sample from
(\ref{eq:posterior}). 
The proposal $(Z^*,K^*,P^*)$ is accepted with probability
\begin{equation}
  \label{eq:acceptance}
  \min \left(1,
    \frac{{\Pi _n}\left( {Z^*,K^*,P^*|A} \right)}
    {{\Pi _n}\left( {Z,K,P|A} \right)}
    \frac{\Pi_{prop}(Z,K,P|Z^*,K^*,P^*)}
    {\Pi_{prop}(Z^*,K^*,P^*|Z,K,P)}\right)
\end{equation}
where $\Pi_{prop}$ denotes the density function of the proposal
distribution and $(Z,K,P)$ denotes the current iteration. 

To be specific, the proposal distribution is adapted from the allocation sampler developed in
\cite{mcdaid2013improved}. 
For each iteration of the sampler, the proposal distribution first sample 
$(Z,K)$ in the spirit of the allocation sampler, 
then sample $P$ given $(Z,K)$.
The proposal distribution is decomposed into two parts:  
 conditional on the previous draw $(P,Z,K)$ and data matrix $A$,   
\[{\Pi _{prop}}\left( {{Z^*},{K^*},{P^*}|Z,K,P,A} \right)
  \propto
  {\Pi _{prop}}\left( {{P^*}|{Z^*},A} \right)
  {\Pi _{prop}}\left( {{Z^*},{K^*}|Z,K,P,A} \right)\]
where $P_{ab}^*|{Z^*},A \distras{ind} 
Beta\left( {{O_{ab}^*} + 1, 
    {n_{ab}^*} - {O_{ab}^*} + 1} \right)$
with $O_{ab}^* \equiv O_{ab} (Z^*)$ and
$n_{ab}^* \equiv n_{ab}(Z^*)$, 
and $(Z^*,K^*)|(Z,K,P,A)$ are simulated in the spirit of the
allocation sampler developed in
\cite{mcdaid2013improved,nobile2007bayesian}.

The proposal distribution of $(Z^*,K^*)|(Z,K,P,A)$ follows the allocation sampler 
of \cite{mcdaid2013improved} but it is different 
in the way that connectivity probability matrix $P$ is involved and
used for likelihood evaluation.
In contrast, the allocation sampler of \cite{mcdaid2013improved}
explores the $(Z,K)$ space with $P$ marginalized out. 
Details of the posterior sampler are in the Supplement. 

The expectation of the proposal distribution $\Pi_{prop}(P^*|(Z^*,A))$ is the
ordinary block constant least squares estimator which is widely used
to estimate the connectivity probability matrix in the literature
\citep[see][for instance]{gao2015rate,klopp2017oracle,van2018bayesian}.
As the proposal density matches the likelihood component
$\Pi(A|P^*,Z^*)$, the acceptance rate is a product of prior density
ratios and proposal density ratios.

\subsection{Posterior Inference}
\label{sec:posterior-inference}
Under the 0-1 loss function $\ell(k,k_0)=1_{k=k_0}$, the Bayes estimate of $K$ is its posterior mode. 
As in the Metropolis-Hastings sampler, $K$ communities may contain empty 
communities, we compute the effective number of communities based on samples of $Z$. 

The community assignment is identified up to a label switching. 
In our matrix formulation, the assignment $Z$ is identified up to a
column permutation. 
That is, $ZZ^T$ is invariant to column permutations. 
If the $(i,j)^{th}$ entry of $ZZ^T$ is 1, node $i$ and node $j$ are classified into the
same community by $Z$. 
In addition, the node-wise connectivity $\theta$ is also identified without relabelling concerns.
With the 0-1 loss function $\ell(Z,Z_0) = 1_{({ZZ^T=Z_0Z_0^T})}$, Bayes estimate
of $Z$ is its posterior mode. 
To pin down the posterior mode of $Z$, we can find the posterior mode of $ZZ^T$ and the
corresponding $Z$ is the posterior mode of $Z$.

\section{Numerical Experiments} 
\label{sec:simulations}

Section \ref{sec:cons-est-k} presents asymptotic properties of Bayesian SBM with
diagonally dominant priors which is henceforth abbreviated as ``DD-SBM''. 
This section assesses finite sample properties of DD-SBM under various settings.  


\subsection{Simulation design}  
\label{sec:baseline}

We perform simulation studies for different 
configurations of the number of communities, network size, and overall sparsity of connectivity.   
In particular, we choose $(k_0, n, \rho) \in \{3,5,7\} \times \{50, 75\} \times \{\frac{1}{2}, 1\}$,
and for each $(k_0, n, \rho)$ configuration,   
100 networks are generated from $SBM(Z_0, \rho P^0, n, k_0)$.


To control the source of variation in the synthetic networks, 
the 100 networks share the same community structure $Z_0$ 
where nodes are deterministically and uniformly assigned to $k_0$ communities; 
the 100 networks also share the same connectivity  matrix 
$\rho P^0$.
The randomness in the 100 synthetic networks is only from the
stochastic generation of Bernoulli trials of  $SBM(Z_0, \rho P^0, n, k_0)$.  

We choose the following cases for $P^0$. 
\begin{itemize}
\item Case 1: $  P^0 = 0.6 \times I_{k_0} + 0.2 \times 1_{k_0}1_{k_0}^T$, 
\item Case 2:  $  P^0 = 0.2 \times I_{k_0} + 0.6 \times 1_{k_0}1_{k_0}^T$,  
\item Case 3:  $  P^0 = 0.4 \times I_{k_0} + 0.4 \times 1_{k_0}1_{k_0}^T$,  
\item Case 4: $  P^0 =  0.2 \times I_{k_0} + 0.2 \times 1_{k_0}1_{k_0}^T
  + 0.4 \times 1_{k_0, \lceil{k_0 /2 \rceil{ }} }1_{k_0, \lceil{k_0 /2 \rceil{ }} }^T$, 
\end{itemize}
where $I_k$ denotes identity matrix of rank $k$,   
$1_k$ denotes the $k-$dimensional vector of ones,
and $1_{n,k}$ denotes the $n-$dimensional vector
with the first $k$ elements being 1 and
the rest $(n-k)$ elements being 0.

In the four cases, within community connectivity probabilities are all 0.8.
For simplicity, the between community connectivity probabilities are the same for Case 1-3; 
in Case 1,  cross community connectivity is weak;  
in Case 2,  cross community connectivity is strong;
and in Case 3, cross community connectivity is medium.  
Case 4 combines the structure of Case 1 and Case 3 and
half of the cross community connectivity is strong. 

The reasons for choosing $n \in \{50, 75\}$ are as follows.
Firstly, many networks in natural and social sciences are often of moderate size.
Secondly, 
asymptotically consistent estimators can perform poorly when sample size
is moderate. 
It is more informative to compare methods for networks of moderate size
than that for networks with thousands of nodes.
Thirdly, MCMC algorithms are computationally expensive, and
the computation bottleneck prevents us from networks with more than thousands of nodes. 

As the number of parameters in the SBM grows in the order of $O(k^2_0)$,
the difficulty of community detection increases as $k_0$ grows.
The case of $k_0=7$ imitates the situation of many communities, while
the cases of $k_0 \in \{3,5\}$ imitate networks with moderately many communties.

\subsection{Simulation results}
\label{sec:simulation-results}

For comparison, we also implement Bayesian SBM with 
the Nowicki and Snijders' prior \citep{nobile2007bayesian, geng2019probabilistic}, 
composite likelihood BIC method \citep{saldana2017many},
and network cross-validation \citep{chen2018network}.   
Two posterior samplers for the Nowicki and Snijders' prior are available in the literature: 
the allocation sampler of \cite{mcdaid2013improved},
and the MFM adapted MCMC algorithm of \cite{geng2019probabilistic}.
We use the code provided in the supplementary materials of \cite{geng2019probabilistic} and
choose default values for the hyperparameters in their algorithm.
The Bayesian SBM of  \cite{geng2019probabilistic, mcdaid2013improved} is henceforth denoted as ``c-SBM'' (Bayesian SBM with conjugate priors).
\cite{saldana2017many}  propose composite likelihood BIC to choose the number of
communities, and this method is henceforth 
denoted as ``CLBIC''. 
\cite{chen2018network} design a cross-validation strategy to choose the number of
communities for SBM, 
and it is henceforth denoted as ``NCV''.

\begin{table}[H]
\centering
\begin{tabular}{cccllllllll}
  \toprule
 & && \multicolumn{2}{c}{Case 1} & \multicolumn{2}{c}{Case 2} & \multicolumn{2}{c}{Case 3} & \multicolumn{2}{c}{Case 4} \\
  \cmidrule(lr){4-5}\cmidrule(lr){6-7}\cmidrule(lr){8-9} \cmidrule(lr){10-11}  
$k_0$& $n$ & Method & $\rho=\mbox{\tiny $\frac12$}$ & $\rho=1$ & $\rho=\mbox{\tiny $\frac12$}$ & $\rho=1$ & $\rho=\mbox{\tiny $\frac12$}$ & $\rho=1$ & $\rho=\mbox{\tiny $\frac12$}$ & $\rho=1$ \\ 
  \midrule
  \multirow{8}{*}{$3$}  
  &\multirow{4}{*}{$50$}  
    & DD-SBM & \te{1.8}{1.3} & \te{1.9}{1.9} & \te{1.8}{-1.6} & \te{1.3}{0.0} & \te{0.3}{0.1}
  & \te{2.0}{-1.9} & \te{0.3}{0.1} & \te{1.0}{-0.6}\\ 
    && c-SBM & \te{0.8}{-0.5} & \te{1.9}{-1.9} & \te{1.9}{-1.8} & \te{1.0}{-1.0} & \te{0.2}{-0.0} & \te{1.9}{-1.9} & \te{0.6}{-0.1} & \te{0.9}{-0.8} \\
    && CLBIC & \te{0.5}{-0.2} & \te{1.3}{-1.2} & \te{1.3 }{-1.2} & \te{1.3}{-1.1} & \te{0.0}{0.0} & \te{1.4}{-1.3} & \te{0.6}{-0.3} & \te{1.0}{-0.9}  \\ 
    && NCV & \te{0.9}{-0.6} & \te{2.0}{-2.0} & \te{2.0}{-2.0} & \te{2.0}{-2.0}  & \te{0.0}{0.0} & \te{2.0}{-2.0} & \te{0.9}{-0.3} & \te{0.9}{-0.8}\\ 
 
 \cmidrule(lr){3-11}
  & \multirow{4}{*}{$75$}  
    &  DD-SBM & \te{1.0}{0.5} & \te{2.0}{-1.9} & \te{1.6}{-1.1} & \te{1.1}{-0.6} & \te{0.1}{0.0} & \te{1.9}{-1.9} & \te{0.2}{0.0} & \te{0.9}{-0.7}\\ 
    && c-SBM & \te{0.5}{-0.1} & \te{2.0}{-1.9} & \te{1.6}{-1.3} & \te{1.0}{-1.0} & \te{0.3}{0.0} & \te{1.8}{-1.6} & \te{0.4}{0.0} & \te{0.9}{-0.8} \\ 
    && CLBIC & \te{0.0}{0.0}  & \te{1.0}{-1.0} & \te{0.9 }{-0.8} & \te{1.0}{-0.9} & \te{0.0}{0.0}  & \te{1.0}{-1.0} & \te{0.0}{0.0} & \te{1.0}{-0.9}  \\ 
    && NCV & \te{0.1}{0.0} & \te{2.0}{-2.0} & \te{1.9}{-1.9} & \te{2.0}{-1.9} & \te{0.0}{0.0} & \te{2.0}{-2.0} & \te{0.0}{0.0} & \te{1.0}{-0.9}\\
  \cmidrule(lr){2-11}

  \multirow{8}{*}{$5$}  
  &\multirow{4}{*}{$50$}  
   &  DD-SBM & \te{3.0}{-2.5} & \te{3.9}{-3.9} & \te{3.9}{-3.8} & \te{2.3}{-2.0} & \te{1.2}{0.7} & \te{4.0}{-4.0} & \te{3.6}{-3.6} & \te{2.8}{-2.7}\\ 
  && c-SBM & \te{3.7}{-3.7} & \te{3.9}{-3.9} & \te{4.0}{-4.0} & \te{3.0}{-3.0} & \te{1.4}{-1.0} & \te{3.9}{-3.9} & \te{3.8}{-3.7} & \te{2.9}{-2.9} \\ 
  && CLBIC & \te{3.1}{-3.1} & \te{3.4}{-3.4} & \te{3.3}{-3.3} & \te{3.5}{-3.4} & \te{1.9}{-1.6} & \te{3.4}{-3.3} & \te{3.2}{-3.2} & \te{2.9}{-2.8}  \\ 
  && NCV    & \te{4.0}{-4.0}   & \te{4.0}{-4.0}    & \te{4.0}{-4.0} & \te{4.0}{-4.0} & \te{2.0}{-1.5}  & \te{4.0}{-4.0} & \te{4.0}{-4.0} & \te{3.2}{-3.0}
  \\ \cmidrule(lr){3-11}

  &\multirow{4}{*}{$75$}  
   &  DD-SBM & \te{2.0}{-1.1} & \te{3.9}{-3.9} & \te{3.9}{-3.9} & \te{2.6}{-2.4} & \te{0.5}{0.0} & \te{4.0}{-4.0} & \te{2.3}{-2.0} & \te{2.9}{-2.8}\\ 
   && c-SBM & \te{2.7}{-2.5} & \te{4.0}{-4.0} & \te{4.0}{-4.0} & \te{3.0}{-3.0} & \te{0.8}{-0.3} & \te{4.0}{-3.9} & \te{2.3}{-2.0} & \te{2.9}{-2.9} \\ 
   && CLBIC & \te{2.6}{-2.5} & \te{3.0}{-3.0} & \te{3.0 }{-3.0} & \te{2.9}{-2.9} & \te{0.0}{0.0}  & \te{3.0}{-3.0} & \te{2.8}{-2.8} & \te{2.7}{-2.7}  \\ 
   && NCV    & \te{3.9}{-3.8}   & \te{4.0}{-4.0}    & \te{4.0}{-4.0} & \te{3.9}{-3.9} & \te{0.0}{0.0} & \te{4.0}{-4.0} & \te{3.9}{-3.8} & \te{2.7}{-2.6}\\
  \cmidrule(lr){2-11} 
  
  \multirow{8}{*}{$7$}  
  & \multirow{4}{*}{$50$}  
   &  DD-SBM & \te{5.6}{-5.5} & \te{5.9}{-5.9} & \te{5.9}{-5.9} & \te{4.1}{-3.9} & \te{3.5}{-3.1} & \te{6.0}{-6.0} & \te{6.0}{-6.0} & \te{4.6}{-4.5}\\ 
   && c-SBM & \te{5.9}{-5.9} & \te{5.9}{-5.9} & \te{6.0}{-6.0} & \te{5.1}{-5.0} & \te{4.8}{-4.7} & \te{6.0}{-5.9} & \te{5.9}{-5.9} & \te{5.0}{-4.9} \\ 
   && CLBIC & \te{5.2}{-5.2} & \te{5.3}{-5.3} & \te{5.3}{-5.3} & \te{5.5}{-5.4} & \te{4.9}{-4.8} & \te{5.3}{-5.3} & \te{5.3}{-5.3} & \te{4.9}{-4.8}  \\ 
   && NCV    & \te{6.0}{-6.0}   & \te{6.0}{-6.0}    & \te{6.0}{-6.0} & \te{6.0}{-6.0} & \te{6.0}{-6.0}  & \te{6.0}{-6.0} & \te{6.0}{-6.0} & \te{5.6}{-5.5}
\\ \cmidrule(lr){3-11}

  &\multirow{4}{*}{$75$}  
   &  DD-SBM & \te{4.7}{-4.6} & \te{6.0}{-6.0} & \te{6.0}{-5.9} & \te{4.4}{{-4.3}} & \te{2.0}{-1.4} & \te{6.0}{-5.9} & \te{5.5}{-5.5} & \te{4.8}{-4.8}\\ 
   && c-SBM & \te{5.4}{-5.4} & \te{5.9}{-5.9} & \te{5.9}{-5.9} & \te{5.0}{-5.0} & \te{2.6}{-2.3} & \te{5.9}{-5.9} & \te{5.4}{-5.3} & \te{5.0}{-5.0} \\ 
   && CLBIC & \te{4.9}{-4.8} & \te{5.0}{-5.0} & \te{5.0 }{-5.0} & \te{4.9}{-4.8} & \te{3.5}{-3.4} & \te{5.0}{-5.0} & \te{5.0}{-5.0} & \te{4.7}{-4.7}  \\ 
   && NCV    & \te{6.0}{-6.0}   & \te{6.0}{-6.0}    & \te{6.0}{-6.0} & \te{6.0}{-6.0} & \te{3.5}{-3.2} & \te{6.0}{-6.0} & \te{6.0}{-6.0} & \te{4.8}{-4.8}\\
  \bottomrule 
\end{tabular}
\caption{RMSE and bias of $\hat{K}$ with bias in parentheses.}  
\label{tab:compare-methods-all-reduced}
\end{table}
 


Compared with c-SBM, DD-SBM achieves similar accuracy across different configurations.
To be specific, 
when $k_0=3$, DD-SBM tends to over-estimate the number of communities; 
when $\rho = \frac{1}{2}$ and $k_0 \in \{5,7\}$, DD-SBM is slightly more accurate
than c-SBM in Case 1 and 3 and similarly accurate to c-SBM in Case 2 and 4.
When the posterior samples of connectivity matrix of c-SBM are also diagonally dominant, 
c-SBM is essentially DD-SBM. Therefore, it is reasonable to expect DD-SBM
and c-SBM have similar accuracy in networks generated from diagonally dominant SBM. 

Compared with CLBIC, DD-SBM is less accurate in most cases.
This is due to the design of $P^0$ in Case 1 - 3, 
such that the working likelihood of CLBIC is close to the true likelihood.
In Case 4, the true likelihood is more complicated than the working
likelihood of CLBIC, 
and the advantage of CLBIC over DD-SBM is less obvious. 

Compared with NCV, DD-SBM is more accurate in most cases. 
To be specific, 
when $k_0=3$ and $\rho = \frac{1}{2}$,  
DD-SBM tends to over-estimate the number of communities;
in other configurations, DD-SBM is more accurate than NCV. 

Case 2 is the most difficult as the between community connectivity probability is
very close to within community connectivity probability.
Indeed, the methods nearly uniformly choose one big community, 
except that CLBIC sometimes chooses two communities.

To assess the membership assignment accuracy, we use the
Hubert-Arabie adjusted Rand index \citep{hubert1985comparing, rand1971objective} to measure the
agreement between two clustering assignments. 
The index is expected to be 0 if two independent assignments are compared, and
is 1 if two equivalent assignments are compared. 
Though the adjusted Rand index tends to capture the disagreement among large clusters,
community sizes in our simulation study are about the same and
the adjusted Rand index is still a meaningful metric. 

\begin{table}[H]
\centering
\begin{tabular}{ccccccccccc}
  \toprule 
  & & & \multicolumn{2}{c}{Case 1} & \multicolumn{2}{c}{Case 2}
  & \multicolumn{2}{c}{Case 3} & \multicolumn{2}{c}{Case 4} \\
    \cmidrule(lr){6-7}\cmidrule(lr){8-9}\cmidrule(lr){10-11} \cmidrule(lr){4-5}  
 $k_0$ & $\rho$ & $n$ & DD-SBM & c-SBM & DD-SBM & c-SBM & DD-SBM & c-SBM & DD-SBM & c-SBM \\ 
  \midrule
     \multirow{4}{*}{$3$}  
  & \multirow{2}{*}{$\frac{1}{2}$}
    & $50$ & 0.62 & 0.63 & 0.00 & 0.00 & 0.05 & 0.01 & 0.37 & 0.42 \\ 
  &
    &$75$ & 0.87 & 0.91 & 0.00 & 0.00 & 0.12 & 0.05 & 0.47 & 0.51 \\ 
  & \multirow{2}{*}{$1$}
    &$50$ & 0.97 & 0.98 & 0.01 & 0.01 & 0.86 & 0.86 & 0.56 & 0.58 \\ 
  &&
   $75$  &  0.99 & 0.99 & 0.03 & 0.03 & 0.96 & 0.97 & 0.58 & 0.57 \\ \cmidrule(lr){2-11}  
     \multirow{4}{*}{$5$}  
  & \multirow{2}{*}{$\frac{1}{2}$}
&$50$ & 0.10 & 0.03 & 0.00 & 0.00 & 0.01 & 0.00 & 0.20 & 0.22 \\ 
  && $75$& 0.25 & 0.11 & 0.00 & 0.00 & 0.01 & 0.00 & 0.28 & 0.30 \\ 
  & \multirow{2}{*}{$1$}
& $50$ &  0.83 & 0.86 & 0.00 & 0.00 & 0.06 & 0.03 & 0.33 & 0.34 \\ 
  &&
  $75$  &  0.94 & 0.99 & 0.00 & 0.00 & 0.32 & 0.17 & 0.35 & 0.36 \\ \cmidrule(lr){2-11}  
     \multirow{4}{*}{$7$}  
  &\multirow{2}{*}{$\frac{1}{2}$}
    &$50$ &  0.03 & 0.00 & 0.00 & 0.00 & 0.01 & 0.00 & 0.12 & 0.12 \\ 
  && $75$&  0.07 & 0.01 & 0.00 & 0.00 & 0.00 & 0.00 & 0.19 & 0.20 \\ 
  & \multirow{2}{*}{$1$}
    &$50$ &  0.24 & 0.12 & 0.00 & 0.00 & 0.01 & 0.00 & 0.24 & 0.24 \\ 
  &&
    $75$&  0.57 & 0.43 & 0.00 & 0.00 & 0.04 & 0.02 & 0.27 & 0.27 \\ 
   \bottomrule
\end{tabular}
\caption{Adjusted Rand index} 
\label{tab:compare-accuracy}
\end{table}

Given a synthetic network $A$ and draws from the posterior distribution $\Pi (\cdot| A)$,
we can compute the adjusted Rand index of 
posterior draws of $Z$ against $Z_0$ and use their mean as the accuracy metric
for $\Pi (\cdot| A)$.  
Like the adjusted Rand index for two clustering assignments, the averaged index assesses 
the agreement of the posterior distribution of $Z$ against the truth $Z_0$. 

Table \ref{tab:compare-accuracy} presents the average of adjusted Rand indices of  the 100 synthetic networks 
under different $(k_0,\rho,n)$ configurations in the four cases. 
Overall, the average adjusted Rand index of DD-SBM is similar to 
that of c-SBM.
This echoes the similar estimation accuracy of $k$ of DD-SBM and c-SBM,
as community detection is highly sensitive to the number of communities.  
When $\rho = 1/2$ and $k_0\in \{5,7\}$, DD-SBM is slightly better than c-SBM in Case 1 and 3. 
When data is less informative, the regularity in the prior of DD-SBM improves
estimation accuracy over c-SBM.
The advantage disappears in Case 2 and 4
where cross community connectivity is close to within community connectivity.  


\section{Sparse Networks}
\label{sec:sparse-networks}

The framework in Section \ref{sec:cons-est-k} can be extended to
sparse networks whose overall connectivity probability shrinks to 0 as network 
size increases \citep[e.g.][]{klopp2017oracle, gao2018minimax}. 
We state the posterior contraction rates and the posterior consistency
results for those sparse networks as follows.
Their proofs follow exactly the same argument 
except that the derivations involve the sparse factor $\rho_n$. 

\begin{theorem}
  \label{thm:posterior-contraction-sparse}
  Suppose adjacency matrix $A \in \{0,1\}^{n \times n}$ is generated from the SBM with 
  $\theta _n^0 = {\rho _n}T(Z_0 P^0Z_0^T)$, 
  $\log (k_0)/n \precsim \rho_n \precsim 1$, 
  $P^0 \in \Theta_{k_0,\delta_0}$ 
  for some $k_0 \precsim \sqrt{n} $
  and $\delta_0 > 0$, and the number of zero and one entries of $T(Z_0P^0Z_0^T)$ is
  at most $O(n^2\varepsilon_n)$ where
  $\varepsilon _n^2 \asymp \frac{\log(k_0)}{n}$.  
  The prior $\Pi_n$ satisfies Assumption \ref{assumption:prior}.
  Then, for all sufficiently large $M$,   
  \[{\bP_{0,n}}{\Pi _n}
    \left( {\theta: 
        ||\theta  - \theta _n^0||_\infty  \ge 
        M \varepsilon _n} | A \right) \to 0.\]
\end{theorem}

The posterior contraction rate in Theorem \ref{thm:posterior-contraction-sparse}
is independent of the sparsity level. 
In contrast, $L_2$ minimax rates of error derived in \cite{klopp2017oracle, gao2018minimax}
are proportional to the sparsity level.
We conjecture that $L_\infty$ minimax rates of error are also proportional to the sparsity level. 
It is likely that the posterior contraction rate in Theorem \ref{thm:posterior-contraction-sparse} 
is sub-optimal. 

\begin{theorem}
  \label{thm:main-sparse}
  Under the same assumptions of Theorem \ref{thm:posterior-contraction-sparse}
  except that the sparsity level satisfies $\log (k_0)/n \precsim \rho^2_n \precsim 1$,
  then 
  \[{\bP_{0,n}}\left[ {{\Pi _n} \left(
          \{K=k_0\} \cap \{Z = {Z_0}\} |A \right)} \right] \to 1. \] 
\end{theorem}

In the sparse network setting, the diagonal dominance gap also vanishes
at the rate of $\rho_n$. 
Our identification strategy for the number of communities
requires $\rho_n \delta_0 \succsim \varepsilon_n \asymp \sqrt{\log(k_0)/n}$ 
to guarantee consistent community detection. 
In contrast, some work in the sparse network literature 
works for networks with sparser sparsity levels   
\citep[e.g.][for a recent survey]{abbe2017community}. 
The Bayesian model outlined in (\ref{eq:proposed-prior}) may need 
additional modifications to adapt to networks at various sparse levels.

\section{Concluding Remarks}
\label{sec:discussions}

In this paper, we have shown Bayesian SBM can consistently estimate the
number of communities and the membership assignment.
Towards this end, we propose the diagonally dominant \ns '
prior and trade conjugacy of \ns' prior for simpler and clearer
asymptotic analysis. 

In the simulation studies, c-SBM has similar finite sample 
estimation accuracy to DD-SBM.  
We conjecture that c-SBM 
can also consistently estimate the
number of communities and the membership assignment
for networks generated from diagonally dominant SBM.
However, the proof technique adopted in this paper cannot be applied
to c-SBM. 

The price of losing conjugacy is on the computation side. 
The posterior sampler in \cite{geng2019probabilistic}  is much faster
than our allocation sampler as they successfully adapt the idea of MFM
sampler of \cite{miller2018mixture} to the SBM case.
It remains unclear if the MFM idea can be applied to the non-conjugate
case.

\bibliographystyle{imsart-nameyear}
\bibliography{SBM-consistency-bib.bib}

\begin{thebibliography}{29}

\bibitem[\protect\citeauthoryear{Abbe}{2017}]{abbe2017community}
\begin{barticle}[author]
\bauthor{\bsnm{Abbe},~\bfnm{Emmanuel}\binits{E.}}
(\byear{2017}).
\btitle{Community detection and stochastic block models: recent developments}.
\bjournal{The Journal of Machine Learning Research}
\bvolume{18}
\bpages{6446--6531}.
\end{barticle}
\endbibitem

\bibitem[\protect\citeauthoryear{Airoldi et~al.}{2008}]{airoldi2008mixed}
\begin{barticle}[author]
\bauthor{\bsnm{Airoldi},~\bfnm{Edoardo~M}\binits{E.~M.}},
  \bauthor{\bsnm{Blei},~\bfnm{David~M}\binits{D.~M.}},
  \bauthor{\bsnm{Fienberg},~\bfnm{Stephen~E}\binits{S.~E.}} \AND
  \bauthor{\bsnm{Xing},~\bfnm{Eric~P}\binits{E.~P.}}
(\byear{2008}).
\btitle{Mixed membership stochastic blockmodels}.
\bjournal{Journal of machine learning research}
\bvolume{9}
\bpages{1981--2014}.
\end{barticle}
\endbibitem

\bibitem[\protect\citeauthoryear{Barron, Schervish and
  Wasserman}{1999}]{barron1999consistency}
\begin{barticle}[author]
\bauthor{\bsnm{Barron},~\bfnm{Andrew}\binits{A.}},
  \bauthor{\bsnm{Schervish},~\bfnm{Mark~J}\binits{M.~J.}} \AND
  \bauthor{\bsnm{Wasserman},~\bfnm{Larry}\binits{L.}}
(\byear{1999}).
\btitle{The consistency of posterior distributions in nonparametric problems}.
\bjournal{The Annals of Statistics}
\bvolume{27}
\bpages{536--561}.
\end{barticle}
\endbibitem

\bibitem[\protect\citeauthoryear{Bickel and
  Sarkar}{2016}]{bickel2016hypothesis}
\begin{barticle}[author]
\bauthor{\bsnm{Bickel},~\bfnm{Peter~J}\binits{P.~J.}} \AND
  \bauthor{\bsnm{Sarkar},~\bfnm{Purnamrita}\binits{P.}}
(\byear{2016}).
\btitle{Hypothesis testing for automated community detection in networks}.
\bjournal{Journal of the Royal Statistical Society: Series B (Statistical
  Methodology)}
\bvolume{78}
\bpages{253--273}.
\end{barticle}
\endbibitem

\bibitem[\protect\citeauthoryear{Chen and Lei}{2018}]{chen2018network}
\begin{barticle}[author]
\bauthor{\bsnm{Chen},~\bfnm{Kehui}\binits{K.}} \AND
  \bauthor{\bsnm{Lei},~\bfnm{Jing}\binits{J.}}
(\byear{2018}).
\btitle{Network cross-validation for determining the number of communities in
  network data}.
\bjournal{Journal of the American Statistical Association}
\bvolume{113}
\bpages{241--251}.
\end{barticle}
\endbibitem

\bibitem[\protect\citeauthoryear{Gao, Lu and Zhou}{2015}]{gao2015rate}
\begin{barticle}[author]
\bauthor{\bsnm{Gao},~\bfnm{Chao}\binits{C.}},
  \bauthor{\bsnm{Lu},~\bfnm{Yu}\binits{Y.}} \AND
  \bauthor{\bsnm{Zhou},~\bfnm{Harrison~H}\binits{H.~H.}}
(\byear{2015}).
\btitle{Rate-optimal graphon estimation}.
\bjournal{The Annals of Statistics}
\bvolume{43}
\bpages{2624--2652}.
\end{barticle}
\endbibitem

\bibitem[\protect\citeauthoryear{Gao and Ma}{2018}]{gao2018minimax}
\begin{barticle}[author]
\bauthor{\bsnm{Gao},~\bfnm{Chao}\binits{C.}} \AND
  \bauthor{\bsnm{Ma},~\bfnm{Zongming}\binits{Z.}}
(\byear{2018}).
\btitle{Minimax rates in network analysis: Graphon estimation, community
  detection and hypothesis testing}.
\bjournal{arXiv preprint arXiv:1811.06055}.
\end{barticle}
\endbibitem

\bibitem[\protect\citeauthoryear{Geng, Bhattacharya and
  Pati}{2019}]{geng2019probabilistic}
\begin{barticle}[author]
\bauthor{\bsnm{Geng},~\bfnm{Junxian}\binits{J.}},
  \bauthor{\bsnm{Bhattacharya},~\bfnm{Anirban}\binits{A.}} \AND
  \bauthor{\bsnm{Pati},~\bfnm{Debdeep}\binits{D.}}
(\byear{2019}).
\btitle{Probabilistic community detection with unknown number of communities}.
\bjournal{Journal of the American Statistical Association}
\bvolume{114}
\bpages{893--905}.
\end{barticle}
\endbibitem

\bibitem[\protect\citeauthoryear{Ghosal, Ghosh and
  Ramamoorthi}{1999}]{ghosal1999posterior}
\begin{barticle}[author]
\bauthor{\bsnm{Ghosal},~\bfnm{Subhashis}\binits{S.}},
  \bauthor{\bsnm{Ghosh},~\bfnm{Jayanta~K}\binits{J.~K.}} \AND
  \bauthor{\bsnm{Ramamoorthi},~\bfnm{RV}\binits{R.}}
(\byear{1999}).
\btitle{Posterior consistency of {Dirichlet} mixtures in density estimation}.
\bjournal{Ann. Statist}
\bvolume{27}
\bpages{143--158}.
\end{barticle}
\endbibitem

\bibitem[\protect\citeauthoryear{Ghosal, Ghosh and van~der
  Vaart}{2000}]{ghosal2000convergence}
\begin{barticle}[author]
\bauthor{\bsnm{Ghosal},~\bfnm{Subhashis}\binits{S.}},
  \bauthor{\bsnm{Ghosh},~\bfnm{Jayanta~K}\binits{J.~K.}} \AND
  \bauthor{\bparticle{van~der} \bsnm{Vaart},~\bfnm{Aad~W}\binits{A.~W.}}
(\byear{2000}).
\btitle{Convergence rates of posterior distributions}.
\bjournal{Annals of Statistics}
\bvolume{28}
\bpages{500--531}.
\end{barticle}
\endbibitem

\bibitem[\protect\citeauthoryear{Ghosal and van~der
  Vaart}{2007}]{ghosal2007convergence}
\begin{barticle}[author]
\bauthor{\bsnm{Ghosal},~\bfnm{Subhashis}\binits{S.}} \AND
  \bauthor{\bparticle{van~der} \bsnm{Vaart},~\bfnm{Aad}\binits{A.}}
(\byear{2007}).
\btitle{Convergence rates of posterior distributions for noniid observations}.
\bjournal{The Annals of Statistics}
\bvolume{35}
\bpages{192--223}.
\end{barticle}
\endbibitem

\bibitem[\protect\citeauthoryear{Ghosal and van~der
  Vaart}{2017}]{ghosal2017fundamentals}
\begin{bbook}[author]
\bauthor{\bsnm{Ghosal},~\bfnm{Subhashis}\binits{S.}} \AND
  \bauthor{\bparticle{van~der} \bsnm{Vaart},~\bfnm{Aad}\binits{A.}}
(\byear{2017}).
\btitle{Fundamentals of nonparametric Bayesian inference}
\bvolume{44}.
\bpublisher{Cambridge University Press}.
\end{bbook}
\endbibitem

\bibitem[\protect\citeauthoryear{Ghosh, Pati and
  Bhattacharya}{2019}]{ghosh2019posterior}
\begin{barticle}[author]
\bauthor{\bsnm{Ghosh},~\bfnm{Prasenjit}\binits{P.}},
  \bauthor{\bsnm{Pati},~\bfnm{Debdeep}\binits{D.}} \AND
  \bauthor{\bsnm{Bhattacharya},~\bfnm{Anirban}\binits{A.}}
(\byear{2019}).
\btitle{Posterior Contraction Rates for Stochastic Block Models}.
\bjournal{Sankhya A}
\bpages{1--29}.
\end{barticle}
\endbibitem

\bibitem[\protect\citeauthoryear{Hubert and Arabie}{1985}]{hubert1985comparing}
\begin{barticle}[author]
\bauthor{\bsnm{Hubert},~\bfnm{Lawrence}\binits{L.}} \AND
  \bauthor{\bsnm{Arabie},~\bfnm{Phipps}\binits{P.}}
(\byear{1985}).
\btitle{Comparing partitions}.
\bjournal{Journal of classification}
\bvolume{2}
\bpages{193--218}.
\end{barticle}
\endbibitem

\bibitem[\protect\citeauthoryear{Karrer and
  Newman}{2011}]{karrer2011stochastic}
\begin{barticle}[author]
\bauthor{\bsnm{Karrer},~\bfnm{Brian}\binits{B.}} \AND
  \bauthor{\bsnm{Newman},~\bfnm{Mark~EJ}\binits{M.~E.}}
(\byear{2011}).
\btitle{Stochastic blockmodels and community structure in networks}.
\bjournal{Physical review E}
\bvolume{83}
\bpages{016107}.
\end{barticle}
\endbibitem

\bibitem[\protect\citeauthoryear{Klopp, Tsybakov and
  Verzelen}{2017}]{klopp2017oracle}
\begin{barticle}[author]
\bauthor{\bsnm{Klopp},~\bfnm{Olga}\binits{O.}},
  \bauthor{\bsnm{Tsybakov},~\bfnm{Alexandre~B}\binits{A.~B.}} \AND
  \bauthor{\bsnm{Verzelen},~\bfnm{Nicolas}\binits{N.}}
(\byear{2017}).
\btitle{Oracle inequalities for network models and sparse graphon estimation}.
\bjournal{The Annals of Statistics}
\bvolume{45}
\bpages{316--354}.
\end{barticle}
\endbibitem

\bibitem[\protect\citeauthoryear{Lei}{2016}]{lei2016goodness}
\begin{barticle}[author]
\bauthor{\bsnm{Lei},~\bfnm{Jing}\binits{J.}}
(\byear{2016}).
\btitle{A goodness-of-fit test for stochastic block models}.
\bjournal{The Annals of Statistics}
\bvolume{44}
\bpages{401--424}.
\end{barticle}
\endbibitem

\bibitem[\protect\citeauthoryear{Li, Levina and Zhu}{2020}]{li2020network}
\begin{barticle}[author]
\bauthor{\bsnm{Li},~\bfnm{Tianxi}\binits{T.}},
  \bauthor{\bsnm{Levina},~\bfnm{Elizaveta}\binits{E.}} \AND
  \bauthor{\bsnm{Zhu},~\bfnm{Ji}\binits{J.}}
(\byear{2020}).
\btitle{Network cross-validation by edge sampling}.
\bjournal{Biometrika}
\bvolume{107}
\bpages{257--276}.
\end{barticle}
\endbibitem

\bibitem[\protect\citeauthoryear{McDaid et~al.}{2013}]{mcdaid2013improved}
\begin{barticle}[author]
\bauthor{\bsnm{McDaid},~\bfnm{Aaron~F}\binits{A.~F.}},
  \bauthor{\bsnm{Murphy},~\bfnm{Thomas~Brendan}\binits{T.~B.}},
  \bauthor{\bsnm{Friel},~\bfnm{Nial}\binits{N.}} \AND
  \bauthor{\bsnm{Hurley},~\bfnm{Neil~J}\binits{N.~J.}}
(\byear{2013}).
\btitle{Improved {Bayesian} inference for the stochastic block model with
  application to large networks}.
\bjournal{Computational Statistics \& Data Analysis}
\bvolume{60}
\bpages{12--31}.
\end{barticle}
\endbibitem

\bibitem[\protect\citeauthoryear{Miller and Harrison}{2018}]{miller2018mixture}
\begin{barticle}[author]
\bauthor{\bsnm{Miller},~\bfnm{Jeffrey~W}\binits{J.~W.}} \AND
  \bauthor{\bsnm{Harrison},~\bfnm{Matthew~T}\binits{M.~T.}}
(\byear{2018}).
\btitle{Mixture models with a prior on the number of components}.
\bjournal{Journal of the American Statistical Association}
\bvolume{113}
\bpages{340--356}.
\end{barticle}
\endbibitem

\bibitem[\protect\citeauthoryear{Nobile and
  Fearnside}{2007}]{nobile2007bayesian}
\begin{barticle}[author]
\bauthor{\bsnm{Nobile},~\bfnm{Agostino}\binits{A.}} \AND
  \bauthor{\bsnm{Fearnside},~\bfnm{Alastair~T}\binits{A.~T.}}
(\byear{2007}).
\btitle{Bayesian finite mixtures with an unknown number of components: The
  allocation sampler}.
\bjournal{Statistics and Computing}
\bvolume{17}
\bpages{147--162}.
\end{barticle}
\endbibitem

\bibitem[\protect\citeauthoryear{Nowicki and
  Snijders}{2001}]{nowicki2001estimation}
\begin{barticle}[author]
\bauthor{\bsnm{Nowicki},~\bfnm{Krzysztof}\binits{K.}} \AND
  \bauthor{\bsnm{Snijders},~\bfnm{Tom A~B}\binits{T.~A.~B.}}
(\byear{2001}).
\btitle{Estimation and prediction for stochastic blockstructures}.
\bjournal{Journal of the American Statistical Association}
\bvolume{96}
\bpages{1077--1087}.
\end{barticle}
\endbibitem

\bibitem[\protect\citeauthoryear{Peixoto}{2014}]{peixoto2014hierarchical}
\begin{barticle}[author]
\bauthor{\bsnm{Peixoto},~\bfnm{Tiago~P}\binits{T.~P.}}
(\byear{2014}).
\btitle{Hierarchical block structures and high-resolution model selection in
  large networks}.
\bjournal{Physical Review X}
\bvolume{4}
\bpages{011047}.
\end{barticle}
\endbibitem

\bibitem[\protect\citeauthoryear{Rand}{1971}]{rand1971objective}
\begin{barticle}[author]
\bauthor{\bsnm{Rand},~\bfnm{William~M}\binits{W.~M.}}
(\byear{1971}).
\btitle{Objective criteria for the evaluation of clustering methods}.
\bjournal{Journal of the American Statistical association}
\bvolume{66}
\bpages{846--850}.
\end{barticle}
\endbibitem

\bibitem[\protect\citeauthoryear{Saldana, Yu and Feng}{2017}]{saldana2017many}
\begin{barticle}[author]
\bauthor{\bsnm{Saldana},~\bfnm{D~Franco}\binits{D.~F.}},
  \bauthor{\bsnm{Yu},~\bfnm{Yi}\binits{Y.}} \AND
  \bauthor{\bsnm{Feng},~\bfnm{Yang}\binits{Y.}}
(\byear{2017}).
\btitle{How many communities are there?}
\bjournal{Journal of Computational and Graphical Statistics}
\bvolume{26}
\bpages{171--181}.
\end{barticle}
\endbibitem

\bibitem[\protect\citeauthoryear{Schwartz}{1965}]{schwartz1965bayes}
\begin{barticle}[author]
\bauthor{\bsnm{Schwartz},~\bfnm{Lorraine}\binits{L.}}
(\byear{1965}).
\btitle{On {Bayes} procedures}.
\bjournal{Probability Theory and Related Fields}
\bvolume{4}
\bpages{10--26}.
\end{barticle}
\endbibitem

\bibitem[\protect\citeauthoryear{van~der Pas and van~der
  Vaart}{2018}]{van2018bayesian}
\begin{barticle}[author]
\bauthor{\bparticle{van~der} \bsnm{Pas},~\bfnm{SL}\binits{S.}} \AND
  \bauthor{\bparticle{van~der} \bsnm{Vaart},~\bfnm{AW}\binits{A.}}
(\byear{2018}).
\btitle{Bayesian community detection}.
\bjournal{Bayesian Analysis}
\bvolume{13}
\bpages{767--796}.
\end{barticle}
\endbibitem

\bibitem[\protect\citeauthoryear{Wang and Bickel}{2017}]{wang2017likelihood}
\begin{barticle}[author]
\bauthor{\bsnm{Wang},~\bfnm{YX~Rachel}\binits{Y.~R.}} \AND
  \bauthor{\bsnm{Bickel},~\bfnm{Peter~J}\binits{P.~J.}}
(\byear{2017}).
\btitle{Likelihood-based model selection for stochastic block models}.
\bjournal{The Annals of Statistics}
\bvolume{45}
\bpages{500--528}.
\end{barticle}
\endbibitem

\bibitem[\protect\citeauthoryear{Zhao, Levina and
  Zhu}{2011}]{zhao2011community}
\begin{barticle}[author]
\bauthor{\bsnm{Zhao},~\bfnm{Yunpeng}\binits{Y.}},
  \bauthor{\bsnm{Levina},~\bfnm{Elizaveta}\binits{E.}} \AND
  \bauthor{\bsnm{Zhu},~\bfnm{Ji}\binits{J.}}
(\byear{2011}).
\btitle{Community extraction for social networks}.
\bjournal{Proceedings of the National Academy of Sciences}
\bvolume{108}
\bpages{7321--7326}.
\end{barticle}
\endbibitem

\end{thebibliography}

\begin{supplement}
\textbf{Supplement to ``Consistent Bayesian Community Detection''}.
This Supplement contains additional results and proofs in the text. 
\end{supplement}

\newpage

\begin{center}
{\bf \large Supplement to ``Consistent Bayesian Community Detection''}  
\end{center}

The supplement file contains complete proofs for
Lemma \ref{lemma:kid},
\ref{lemma:norm-dominance},
\ref{lemma:entropy},
\ref{lemma:KL-dominance} and
\ref{lemma:prior-mass},
details of the sampler,
and complete simulation results for all configurations. 

\section{Proofs}
\label{sec:proofs}


\paragraph{Proof of Lemma \ref{lemma:kid}}
\begin{proof}
  Suppose $\theta \in \Theta_{k,\delta}$, 
  it suffices to show $\theta \notin \Theta_{k^\prime, \delta^\prime}$
  for all $k^\prime<k$ and $\delta^\prime \ge 0$.
  Now prove the statement by contradiction.

  If $\theta \in \Theta_{k^\prime, \delta^\prime}$ for some
  $k^\prime<k$ and $\delta^\prime \ge 0$, then 
  some nodes from some communities implied by $\theta$ are merged. 
  But by construction of $\Theta_{k,\delta}$, between-community
  connectivity probabilities of $\theta$
  are strictly less than corresponding within community connectivity probabilities.
  Therefore, once merged, the connectivity probabilities of the merged
  block are not identical. This is a contradiction.  
  
\end{proof}

\paragraph{Proof of Lemma \ref{lemma:norm-dominance}}
\begin{proof}
  The Hellinger distance between two Bernoulli random variables satisfies  
  \[\begin{array}{lll}
      {H^2}\left( {{\bP_{\theta _{ij}^0}},{\bP_{\theta _{ij}^1}}} \right)
      &=& \frac{1}{2}\left[ {{{\left( {\sqrt {\theta _{ij}^0}
          - \sqrt {\theta _{ij}^1} } \right)}^2}
          + {{\left( {\sqrt {1 - \theta _{ij}^0}
          - \sqrt {1 - \theta _{ij}^1} } \right)}^2}} \right]\\
      &= &
           \frac{1}{2}\left[ {{{\left( {\frac{1}{2}2|\sqrt {\theta _{ij}^0}
           - \sqrt {\theta _{ij}^1} |} \right)}^2}
           + {{\left( {\frac{1}{2}2|\sqrt {1 - \theta _{ij}^0}
           - \sqrt {1 - \theta _{ij}^1} |} \right)}^2}} \right]\\
      &\ge&
            \frac{1}{4}{\left( {\theta _{ij}^0 - \theta _{ij}^1} \right)^2}
    \end{array}\]
  as $\theta_{ij}^0$ and $\theta_{ij}^1$ are in $[0,1]$,
  $|\sqrt {\theta _{ij}^0} + \sqrt {\theta _{ij}^1} | \le 2 $ and
  $|\sqrt {1-\theta _{ij}^0} + \sqrt {1-\theta _{ij}^1} | \le 2 $.
  
  By independence, ${P_\theta } = { \otimes _{i < j}}{P_{{\theta _{ij}}}}$. 
  Then, the Hellinger distance between $\bP_{\theta^0}$ and $\bP_{\theta^1}$
  satisfies 
  \[\begin{array}{lll}
      {H^2}\left( {{P_{{\theta ^0}}},{P_{{\theta ^1}}}} \right)
      &=&
          2 - 2\prod\nolimits_{i < j}^{} {\left( {1 - \frac{1}{2}
          {H^2}\left( {{P_{\theta _{ij}^0}},{P_{\theta _{ij}^1}}} \right)} \right)} \\
      &\ge &
             2 - 2\prod\nolimits_{i < j}^{} {\left( {1 - 
             \frac{1}{8}{\left( {\theta _{ij}^0 - \theta _{ij}^1} \right)^2}
             } \right)} \\
      &\ge &
             2 - 2\mathop {\min }\nolimits_{i < j}
             \left( {1 -
             \frac{1}{8}{\left( {\theta _{ij}^0 - \theta _{ij}^1} \right)^2}
             } \right)\\
      &=&\frac{1}{4}
          \mathop {\max }\nolimits_{i < j}
          \left( {\theta _{ij}^0 - \theta _{ij}^1} \right)^2      \\
      &=&
            \frac{1}{4}||{\theta ^0} - {\theta ^1}||_\infty ^2. 
    \end{array}\]
\end{proof}

 \paragraph{Proof of Lemma \ref{lemma:entropy}}
\begin{proof}

  Note ${\Theta _{k,\delta_n }} =
  { \cup _{Z \in {\cZ_{n,k}}}} \Theta _{k,\delta_n }^Z$,  
  where $\Theta _{k,\delta_n }^Z
  = \left\{ T(ZPZ^T): P \in {S_{k,\delta_n }} \right\}$ denotes the $Z$
  slice of the parameter space.

  By Lemma \ref{lemma:ID-Z} and the assumption on $\delta_n$ and
  $\varepsilon_n$, node-wise connectivity probability matrix space can
  be simplified via 
  $\left\{ {\theta :||\theta  - {\theta ^0}|{|_\infty } < {\varepsilon _n}} \right\}
  = \left\{ T(Z_0PZ_0^T):||P - {P^0}|{|_\infty } < {\varepsilon _n}
  \right\}$.  
  This relation implies the covering number 
  $N\left( {{\varepsilon _n},\Theta _{k,\delta_n }^Z,
      || \cdot |{|_\infty }} \right)
  \le {\left( {1/{\varepsilon _n}} \right)^{k\left( {k + 1} \right)/2}}$, 
  and then union bound implies the covering number 
  $N\left( {{\varepsilon _n},\Theta _{k,\delta_n }^{},
      || \cdot |{|_\infty }} \right)
  \le
  {k^n}{\left( {1/{\varepsilon _n}} \right)^{k\left( {k + 1}
      \right)/2}}$.
  
  By Lemma \ref{lemma:kid}, $\Theta _{k,\delta }$ are non-overlapping
  for different $k$, then another union bound implies the statement (\ref{eq:entropy}).

\end{proof}
 
\paragraph{Proof of Lemma \ref{lemma:KL-dominance}}

\begin{proof}
  First recall some basic expansions from calculus. For $x_0\in (0,1)$, define
  $f\left( x \right)
  =  - {x_0}\log \frac{x}{{x{ _0}}}
  - \left( {1 - {x_0}} \right)\log \frac{{1 - x}}{{1 - {x_0}}}$ for $x \in [0,1]$. 
  Taylor expand $f(x)$ around $x_0$:
  \[\begin{array}{lll}
      f\left( x \right)
      &=&
          f\left( {{x_0}} \right)
          + {f^\prime}\left( {{x_0}} \right)\left( {x - {x_0}} \right)
          + \frac{1}{2}{f^{\prime \prime}}\left( {{x_0}} \right){\left( {x - {x_0}} \right)^2}
          + O\left( {|x - {x_0}{|^3}} \right)\\
      &=& 
          \frac{1}{{2{x_0}\left( {1 - {x_0}} \right)}}{\left( {x - {x_0}} \right)^2}
          + O\left( {|x - {x_0}{|^3}} \right). 
    \end{array}\]
  For $x_0=0$, the above $f(x) = -\log(1-x)$ with the convention $0 \log 0 =0$.
  Its Taylor expansion around $0$ is $f\left( x \right) =  - \log (1 - x) = x + O\left( {{x^2}} \right)$. 
  For $x_0=1$, the above $f(x) = -\log(x)$ also with the convention $0 \log 0 =0$. 
  Its Taylor expansion around $1$ is
  $f\left( x \right) =  - \log (x) = 1 - x + O\left( {{{\left( {1 - x} \right)}^2}} \right)$. 

  With $||\theta-  \theta^0||_{\infty} \le \varepsilon_n $ and
  the assumption on $\theta^0$, 
  expand KL divergence at $\theta^0$, 
  \[\begin{array}{lll} 
      KL\left( {{\bP_{{ \theta ^0}}},{\bP_\theta }} \right)
      &=&
          - \sum\nolimits_{i < j:\theta _{ij}^0 > 0}
          {\theta _{ij}^0\log \frac{{{\theta _{ij}}}}{{ \theta _{ij}^0}}}
          - \sum\nolimits_{i < j:\theta _{ij}^0 < 1} 
          {\left( {1 - \theta _{ij}^0} \right)
          \log \frac{{1 - {\theta _{ij}}}}{{1 - \theta _{ij}^0}}} \\
      &\le&
            \left( {{N_0} + {N_1}} \right)\left( {{\varepsilon _n}
            + O\left( {\varepsilon _n^2 } \right)} \right)
            + \frac{{n\left( {n - 1} \right)}}{2 }C_0^{ - 1}
            \left( {\varepsilon _n^2
            + O\left( {|{\varepsilon _n}{|^3}} \right)} \right)\\
      &\precsim&
                 {n^2}\varepsilon _n^2/{C_0}
    \end{array}\]
  where ${N_0} = \# \left\{ {\left( {i,j} \right):
      \theta _{ij}^0 = 0,i < j} \right\}$ denotes the number of zero
  entries in $\theta^0$,  and 
  ${N_1} = \# \left\{ {\left( {i,j} \right):
      \theta _{ij}^0 = 1,i < j} \right\}$ denotes the number of one
  entries in $\theta^0$.

  To bound $V_{2,0}$, note the Taylor expansion of $f\left( x \right) = \log \frac{x}{{1 - x}}$
  around $x_0 \in (0,1)$ satisfies
  $f\left( x \right)
  = \log \frac{x}{{1 - x}} = \log \frac{{{x_0}}}{{1 - {x_0}}}
  + \frac{1}{{{x_0}\left( {1 - {x_0}} \right)}}\left( {x - {x_0}} \right) 
  + O\left( {{{\left( {x - {x_0}} \right)}^2}} \right)$.
  
  By independence of different entries
  and with $||\theta-  \theta^0||_{\infty}
  \le \varepsilon_n$, KL variation can be  
  bounded similarly by an expansion of $f(x)=\log(x/(1-x))$: 
  \[\begin{array}{lll} 
      {V_{2,0}}\left( {{\bP_{{ \theta ^0}}},{\bP_\theta }} \right)
      &=&
          {\bP_0}\left\{ {{{\left[ {  \sum\limits_{i < j}^{}
          {\left( {{A_{ij}}\log \frac{{{\theta _{ij}}}}{{\theta _{ij}^0}}
          + \left( {1 - {A_{ij}}} \right)
          \log \frac{{1 - {\theta _{ij}}}}
          {{1 - \theta _{ij}^0}}} \right)}
          + KL\left( {{\bP_{{ \theta ^0}}},{\bP_\theta }} \right)} \right]}^2}} \right\}\\
      &=&
          \sum\limits_{i < j}^{} {{\bP_0}\left\{ {{{\left[
          {\left( {{A_{ij}}\log \frac{{{\theta _{ij}}}}{{\theta _{ij}^0}}
          + \left( {1 - {A_{ij}}} \right)
          \log \frac{{1 - {\theta _{ij}}}}{{1 - \theta _{ij}^0}}} \right)
          + KL\left( {{\bP_{\theta _{ij}^0}},{\bP_{{\theta _{ij}}}}} \right)} \right]}^2}} \right\}} \\
      &=& 
          \sum\nolimits_{i < j}^{}
          {\theta _{ij}^0\left( {1 -  \theta _{ij}^0} \right)
          {{\left( {\log \frac{{{\theta _{ij}}}}{{1 - {\theta _{ij}}}} 
          - \log \frac{{\theta _{ij}^0}}{{1 - \theta _{ij}^0}}} \right)}^2}} \\
      &\precsim&
                 \sum\nolimits_{i < j}^{}
                 \frac{1}
                 {\theta _{ij}^0\left( {1 - \theta _{ij}^0} \right)}
                 \varepsilon _n^2
      \\
      &\precsim&
                 {n^2}\varepsilon _n^2/{C_0}
    \end{array}\]

\end{proof}

\paragraph{Proof of Lemma \ref{lemma:prior-mass}}
\begin{proof}

  By the dependence assumption made in Assumption
  \ref{assumption:prior},
  the prior mass has the factorization
  \begin{equation}
    \label{eq:prior-mass-fac}
    {\Pi _n}\left( {P:
        ||P -  {P^0}|{|_\infty } < C_0{\varepsilon _n} }|K=k_0 \right)
    {\Pi _n}\left( {Z = {Z_0}|K = {k_0}} \right)
    {\Pi _n}\left( {K = {k_0}} \right).  
  \end{equation}
  Next, we bound individual components of (\ref{eq:prior-mass-fac}) respectively.
  
To bound the first component of (\ref{eq:prior-mass-fac}), 
the conditional indepence of the off-diagonal entries of $P$ on the diagonal entries of $P$ suggests the 
following factorization,  
\[\begin{array}{ll}
    &  {\Pi _n}\left( {P:
      ||P -  {P^0}|{|_\infty } < C_0 {\varepsilon _n} }
      |K=k_0 \right) \\ 
    =& {\Pi _n}\left( {
        \bigcap _{1\le a \le b \le k_0} E_{n,ab}}|K=k_0 \right) \\
    = & \prod_{1 \le a \le k_0} \left\{ 
        \int_{E_{n,aa}} \left[
        \prod \nolimits_{1\le a < b \le k_0} 
        {\Pi _n}\left(E_{n,ab}|\{P_{aa}\}, K=k_0\right) \right]
       d \Pi_n(P_{aa}|K=k_0 ) \right\} 
  \end{array}\]
where $E_{n,ab}= \{ P_{ab}: |P_{ab} -  {P_{ab}^0}| < C_0 {\varepsilon _n} \}$.
As ${\varepsilon _n}  =o(1)$ and $P^0\in S_{k_0,\delta_0}$,
(conditional) prior density of 
$P_{ab}$ is positive on $E_{n,ab}$ for all $a,b \in [k_0]$. 

By Assumption  \ref{assumption:prior} (2), for $a<b \in [k_0]$, the prior probability 
$\Pi_n(E_{n,ab}|\{P_{aa}\},K=k_0)
\ge |E_{n,ab}| \min\limits_{P_{ab} \in E_{n,ab}} \pi_n(P_{ab}|\{P_{aa}\},K=k_0,\delta)
\succsim \varepsilon _n e^{-C\log(n) (P_{aa}\wedge P_{bb}) } $
for some universal constant $C$.  
As $P_{aa}\in E_{n,aa}$ for $a \in [k_0]$,   $P_{aa}\wedge P_{bb} 
\le (P^0_{aa}\wedge P^0_{bb}) + C_0 \varepsilon _n 
\le ||P^0||_{\infty} +  C_0\varepsilon _n$, which
gives a bound independent of $\{P_{aa}\}$.   

Similarly, Assumption  \ref{assumption:prior} (2)  implies 
$\Pi_n(E_{n,aa}|K=k_0)
\succsim \varepsilon _n
e^{-C\log(n) ( P^0_{aa} + C_0\varepsilon _n) }$.  

Therefore, combining the bounds for $P_{ab}$'s gives  
\[{\Pi _n}\left( {P:
    ||P -  {P^0}|{|_\infty } < C_0{\varepsilon _n} }|K=k_0 \right) 
\succsim e^{C k_0^2 \log (\varepsilon _n)
  -C k_0^2 \log(n) (||P^0||_{\infty}+ C_0 \varepsilon _n) } \]
where  $k_0^2$ has the same order as  $\frac{1}{2}k_0(k_0+1)$ and is used for simpler notation, and the constant $C$ is universal.

As $\varepsilon_n^2 \asymp \log(k_0)/n $ and $1 \succsim \log(k_0)/n $,
$\log(n) \succsim -\log(\varepsilon_n)$. 
As $k_0 \precsim \sqrt{n}$, $k_0^2 \log(n) \precsim n \log(k_0)$.
Then,  ${\Pi _n}\left( {P: 
    ||P -  {P^0}|{|_\infty } < C_0{\varepsilon _n} }|K=k_0 \right)
\succsim
e^{-C n\log(k_0)}$ 
for some constant $C$ dependent on $P^0$.

To bound the second and the third component of (\ref{eq:prior-mass-fac}),
by Assumption \ref{assumption:prior} (3) and (4), 
there exists a universal constant $C$ such that 
${\Pi _n}\left( {Z = {Z_0}|K = {k_0}} \right) \ge e^{-C n \log(k_0)}$ and
${\Pi _n}\left( {K = {k_0}} \right) \ge e^{-C n \log(k_0)}$. 

Note $n^2 \varepsilon_n^2 \asymp  n \log(k_0)$, 
the right hand side of the inequality (\ref{eq:prior-mass}) can be
replaced with $e^{-Cn\log(k_0)}$ and (\ref{eq:prior-mass}) holds for
some constant $C$ dependent on $P^0$.

\end{proof}

\section{Posterior Sampler}
\label{sec:posterior-sampler}

This section presents details of the Metropolis-Hastings algorithm
used to draw posterior samples from (\ref{eq:posterior}). 
The proposal has two stages: in the first stage,
sample $(Z,K)$; in the second stage, sample $P$ given $(Z,K)$.
The first stage is adapted from the allocation sampler \citep{mcdaid2013improved}. 

At $t^{th}$ iteration, 
the proposal
$\Pi_{prop}\left( Z^{*},K^{*}|A, Z^{(t)},K^{(t)},P^{(t)}\right) $ consists of 
the four steps MK, GS, M3 and AE with equal probability $\frac{1}{4}$.
With proposal $(Z^*,K^*)$, sample $P^*|(Z^*,A)$ by independently sampling 
each entry of $P^*$ from the  Beta distribution
$ Beta\left( {{O_{ab}^*} + 1, {n_{ab}^*} - {O_{ab}^*} + 1} \right)$.
With proposal $(P^*,Z^*,K^*)$, the acceptance rates in the allocation sampler
regimes are computed.

\subsection{MK}
\label{sec:mk}
 MK: choose ``add'' or ``delete'' one empty cluster with
  probability $1/2$. 
  If ``add'' move is chosen, randomly pick one community identifier
  from $[K+1]$ for the new empty community and rename the others as
  necessary; 
  if ``delete'' move is chosen, randomly pick one community from $[K]$,
  delete the community if it is empty and abandon the MK move if it is not empty.

  In the step MK, if ``add'' one empty community is chosen,
  accept the proposal with probability
  $\min \left( {1,
      \frac{{{\Pi _n}\left( {{P^*}|{Z^*}} \right)}}
      {{{\Pi _n}\left( {{P^{\left( t \right)}}|{Z^{\left( t \right)}}} \right)}}
      \frac{K}{{{K^*}}}}
    \frac{1}{n+K} \right)$;
  if ``delete'' one empty community is chosen, accept the proposal
    with probability   \[\min \left( {1,
      \frac{{{\Pi _n}\left( {{P^*}|{Z^*}} \right)}}
      {{{\Pi _n}\left( {{P^{\left( t \right)}}|{Z^{\left( t \right)}}} \right)}}
      \frac{K}{{{K^*}}}}
    {(n+K-1)} \right).\]  
  
\subsection{GS}
\label{sec:gs}

GS: relabel a random node.
  First randomly pick $i$ then generate $Z^*(i)$ according to
  $\Pi_{prop}(Z^*(i)=k) \propto \beta(Z^*,A)^{-1} \Pi(Z^*|K^*)$
  where $K^*=K^{(t)}$, the prior probability
  $\Pi(Z^*|K^*)=\int \Pi(Z^*|\alpha,K^*)\Pi(\alpha|K^*)d\alpha
  =\frac{\Gamma(K^*)}{\Gamma(n+K^*)} \prod_{1\le c \le K} \Gamma(n_c ^* +1)$ due to
  multinomial-Dirichlet conjugacy, 
  and $\beta(Z^*,A)
  = \prod_{1\le a \le b \le K}
  \frac{\Gamma(n_{ab}^*+2)}{\Gamma(O_{ab}^* +1)\Gamma(n_{ab}^* -O_{ab}^* +1)}$ is
  the coefficient corresponding to the proposal distribution of $P$.  
  Clearly, $Z^*(j)=Z(j)$ for all $j \neq i \in [n]$.

  In the step GS, suppose node $i$ is chosen and its original label $c_1$ is
  relabeled with $c_2$, then 
  accept the proposal with probability
  $\min \left( 1,
        \frac{{{\Pi _n}\left( {{P^*}|{Z^*}} \right)}}
        {{{\Pi _n}\left( {{P^{\left( t \right)}}|{Z^{\left( t \right)}}} \right)}}
    \right)$. 

\subsection{M3}
\label{sec:m3}

  M3:
  randomly pick two communities $c_1,c_2 \in [K]$,
  reassign nodes $\{i:Z(i)\in \{c_1,c_2 \}\}$ to $\{c_1, c_2 \}$ sequentially
  according to the following scheme.
  Start with $B_0=B_1=\emptyset$ and
  $A_0$ being the sub-network without nodes from community $c_1$ and $c_2$,
  define the assignment 
  $B_h = \{Z^*(x_i)\}_{i=1}^{h-1}$ with $x_i$ being the node index of the $i^{th}$ element in
  $\{i:Z(i)\in \{c_1,c_2 \}\}$,
  define the sub-network $A_h=A_{h-1}\cup \{x_h\}$ by appending one more node,
  define the assignment $Z_{B_h}^{c_j}$ for the sub-network $A_h$
  as the assignment with the node $x_h$ assigned to $c_j$,
  and define the size of communities in the sub-network $A_{h-1}$
  as $\{n_{h,c}\}_{c\in [K]}$. 
  For $i \in [n_c]$, assign the $i^{th}$ node of $\{i:Z(i)\in \{c_1,c_2 \}\}$
  to $c_1$ with probability $p_{B_i}^{c_1}$
  and to $c_2$ with probability $p_{B_i}^{c_2} \equiv 1-p_{B_i}^{c_1}$,
  where $\frac{p_{B_i}^{c_1}}{p_{B_i}^{c_2}}
  = \frac{\Pi(A_i,Z_{B_i}^{c_1},K,P)}{\Pi(A_i,Z_{B_i}^{c_2},K,P)}
  = \frac{\Pi(A_i|P, Z_{B_i}^{c_1})\Pi(P|Z_{B_i}^{c_1},K) \Pi(Z_{B_i}^{c_1}|K) \Pi(K)}
  {\Pi(A_i|P, Z_{B_i}^{c_2})\Pi(P|Z_{B_i}^{c_2},K) \Pi(Z_{B_i}^{c_2}|K) \Pi(K)}
  = \frac{\Pi(A_i|P, Z_{B_i}^{c_1})(n_{i,c_1}+1)}
  {\Pi(A_i|P, Z_{B_i}^{c_2})(n_{i,c_2}+1)}$. 

  To improve mixing, once $c_1$ and $c_2$ are drawn, shuffle 
  $\{i:Z(i)\in \{c_1,c_2 \}\}$ before the sequential reassignment. 
  Therefore, the ordering of node indices in the sequential reassignment is random.

  In the step M3, suppose community $c_1$ and $c_2$ are chosen, then 
  accept the proposal with probability 
  $\min \left( {1,
        \frac{{{\Pi _n}\left( {{P^*}|{Z^*}} \right)}}
        {{{\Pi _n}\left( {{P^{\left( t \right)}}|{Z^{\left( t \right)}}} \right)}}
        \frac{\prod_{i=1}^{n_c}p_{B_i}^{Z(i)}}
        {\prod_{i=1}^{n_c}p_{B_i}^{Z^*(i)}}
        \frac{{\Gamma \left( {n_{{c_1}}^{*}+1} \right)
            \Gamma \left( {n_{{c_2}}^{*}+1} \right)}}
        {{\Gamma \left( {n_{{c_1}}^{(t)} +1} \right)
            \Gamma \left( {n_{{c_2}}^{(t)}+1} \right)}}
        \frac{{\beta \left( {{Z^{\left( t \right)}},A} \right)}}
        {{\beta \left( {{Z^*},A} \right)}}} \right)$,
    where $n_c=n_{c_1}+n_{c_2}$.

\subsection{AE}
\label{sec:ae}
AE: merge two random clusters or split one cluster into two clusters with probability $1/2$. 
  If ``merge'' is chosen, randomly merge two clusters $c_1$ and $c_2$ 
  with $Z^*(i) = c_1$ for all $i\in \{j:Z(j)\in \{c_1,c_2\}\}$ and
  $Z^*(i)=Z(i)$ for all $i \notin \{j:Z(j)\in \{c_1,c_2\}\} $.
  The proposal probability is $\binom{K}{2}^{-1}$. 
  If ``split'' is chosen, randomly pick two cluster identifiers $\{c_1,c_2\}$ from $[K+1]$,
  renaming others' identifiers as necessary, and assign the nodes in cluster $c_1$ to the cluster
  $c_2$ with the random probability $p_c \sim U(0,1)$. 
  By integrating out $p_c$, 
  the proposal probability is
  $ \frac{{\Gamma \left( {{n_{{c_1}}} + 1} \right)
      \Gamma \left( {{n_{{c_2}}} + 1} \right)}}
  {K(K+1){\Gamma \left( {{n_c} + 2} \right)}}$. 

In the step AE, if ``merge'' two communities is chosen, 
  accept the proposal with probability
  $\min \left( {1,\frac{{{\Pi _n}\left( {{P^*}|{Z^*}} \right)}}
        {{{\Pi _n}\left( {{P^{\left( t \right)}}|{Z^{\left( t \right)}}} \right)}}
        \frac{{{K^{\left( t \right)}}}}{{{K^{*}}}} 
        \frac{{\beta \left( {{Z^{\left( t \right)}},A} \right)}}
        {{\beta \left( {{Z^*},A} \right)}}
        \frac{K^*+n}
        {{ {n_{c_1}^{\left( t \right)} + 1} }}} \right)$; 
    if ``split''  is chosen,
    accept the proposal with probability
  $\min \left( {1,\frac{{{\Pi _n}\left( {{P^*}|{Z^*}} \right)}}
        {{{\Pi _n}\left( {{P^{\left( t \right)}}|{Z^{\left( t \right)}}} \right)}}
        \frac{{{K^{\left( t \right)}}}}{{{K^*}}}
        \frac{{\beta \left( {{Z^{\left( t \right)}},A} \right)}}
        {{\beta \left( {{Z^*},A} \right)}}
        \frac{n_{c_1}^{(t)} + 1}{K+n}
        } \right)$.

\section{Complete simulation results}
\label{sec:compl-simul-results}
This section provides complete simulation results.
We choose $(k_0, n, \rho) \in \{3,5,7\} \times \{50, 75\} \times \{\frac{1}{2}, 1\}$,
and for each $(k_0, n, \rho)$ configuration,   
100 networks are generated from $SBM(Z_0, \rho P^0, n, k_0)$. 

To reduce Monte Carlo error and reach reasonable mixing, 
the Metropolis-Hastings algorithm and the allocation sampler collect $2 \times 10^4$ posterior draws for  each synthetic dataset  after discarding first $10^4$ draws as burn-in.
Both algorithms are initialized at $K=2$ and random membership assignment.

\begin{table}[H]
\centering
\begin{tabular}{cccccccccccc}
  \toprule
 & &&& \multicolumn{2}{c}{Case 1} & \multicolumn{2}{c}{Case 2} & \multicolumn{2}{c}{Case 3} & \multicolumn{2}{c}{Case 4} \\
  \cmidrule(lr){5-6}\cmidrule(lr){7-8}\cmidrule(lr){9-10} \cmidrule(lr){11-12}  
$k_0$& $n$ & $\rho$  & Method & Bias & RMSE & Bias & RMSE & Bias & RMSE & Bias & RMSE \\ 
  \midrule
  \multirow{8}{*}{$3$}  & 
  \multirow{8}{*}{$50$}  & 
  \multirow{4}{*}{$\frac{1}{2}$}  
   &  DD-SBM & 1.3 & 1.8 & -1.9 & 1.9 & -1.6 & 1.8 & {0.0} & 1.3 \\ 
  & && c-SBM & -0.5 & 0.8 & -1.9 & 1.9 & -1.8 & 1.9 & -1.0 & 1.0\\ 
  & && CLBIC & {-0.2} & {0.5} & {-1.2} & {1.3}& {-1.2} & {1.3 }& -1.1 & 1.3 \\ 
  &&& NCV & -0.6 & 0.9 & -2.0 & 2.0 & -2.0 & 2.0 & -2.0 & 2.0 
  \\ \cmidrule(lr){4-12}
  & &   \multirow{4}{*}{$1$}   
  & DD-SBM & 0.1 & 0.3 & -1.9 & 2.0 & {0.1} & {0.3} & {-0.6} & 1.0\\ 
  &&& c-SBM & -0.0 & 0.2 & -1.9 & 1.9 & -0.1 & 0.6 & -0.8 & 0.9 \\ 
  &&& CLBIC & {0.0} & {0.0} & {-1.3} & {1.4}& -0.3 & 0.6 & -0.9 & 1.0  \\ 
  &&& NCV & {0.0} & {0.0} & -2.0 & 2.0 & -0.3 & 0.9 & -0.8  & 0.9\\ 
  \cmidrule(lr){4-12}
  
  \multirow{8}{*}{$3$}  & 
  \multirow{8}{*}{$75$}  & 
  \multirow{4}{*}{$\frac{1}{2}$}  
   &  DD-SBM & 0.5 & 1.0 & -1.9 & 2.0 & -1.1 & 1.6 & {-0.6} & 1.1 \\ 
  & && c-SBM & -0.1 & 0.5 & -1.9 & 2.0 & -1.3 & 1.6 & -1.0 & 1.0\\ 
  & && CLBIC & {0.0} & {0.0} & {-1.0} & {1.0}& {-0.8} & {0.9 }& -0.9 & 1.0 \\ 
  &&& NCV & 0.0 & 0.1 & -2.0 & 2.0 & -1.9 & 1.9 & -1.9 & 2.0 
  \\ \cmidrule(lr){4-12}
  & &   \multirow{4}{*}{$1$}   
  & DD-SBM & 0.0 & 0.1 & -1.9 & 1.9 & {0.0} & {0.2} & {-0.7} & 0.9\\ 
  &&& c-SBM & 0.0 & 0.3 & -1.6 & 1.8 & 0.0 & 0.4 & -0.8 & 0.9 \\ 
  &&& CLBIC & {0.0} & {0.0} & {-1.0} & {1.0}& 0.0 & 0.0 & -0.9 & 1.0  \\ 
  &&& NCV & {0.0} & {0.0} & -2.0 & 2.0 & 0.0 & 0.0 & -0.9  & 1.0\\
  \cmidrule(lr){2-12}
  
  \multirow{8}{*}{$5$}  & 
  \multirow{8}{*}{$50$}  & 
  \multirow{4}{*}{$\frac{1}{2}$}  
   &  DD-SBM & -2.5 & 3.0 & -3.9 & 3.9 & -3.8 & 3.9 & {-2.0} & 2.3 \\ 
  & && c-SBM & -3.7 & 3.7 & -3.9 & 3.9 & -4.0 & 4.0 & -3.0 & 3.0\\ 
  & && CLBIC & {-3.1} & {3.1} & {-3.4} & {3.4}& {-3.3} & {3.3 }& -3.4 & 3.5 \\ 
  &&& NCV    & -4.0 & 4.0   & -4.0   & 4.0    & -4.0 & 4.0 & -4.0 & 4.0 
  \\ \cmidrule(lr){4-12}
  & &   \multirow{4}{*}{$1$}   
  & DD-SBM & 0.7 & 1.2 & -4.0 & 4.0 & {-3.6} & {3.6} & {-2.7} & 2.8\\ 
  &&& c-SBM & -1.0 & 1.4 & -3.9 & 3.9 & -3.7 & 3.8 & -2.9 & 2.9 \\ 
  &&& CLBIC & {-1.6} & {1.9} & {-3.3} & {3.4}& -3.2 & 3.2 & -2.8 & 2.9  \\ 
  &&& NCV & {-1.5} & {2.0} & -4.0 & 4.0 & -4.0 & 4.0 & -3.0  & 3.2\\
  
  \cmidrule(lr){3-12} 
  \multirow{8}{*}{$5$}  & 
  \multirow{8}{*}{$75$}  & 
  \multirow{4}{*}{$\frac{1}{2}$}  
   &  DD-SBM & -1.1 & 2.0 & -3.9 & 3.9 & -3.9 & 3.9 & {-2.4} & 2.6 \\ 
  & && c-SBM & -2.5 & 2.7 & -4.0 & 4.0 & -4.0 & 4.0 & -3.0 & 3.0\\ 
  & && CLBIC & {-2.5} & {2.6} & {-3.0} & {3.0}& {-3.0} & {3.0 }& -2.9 & 2.9 \\ 
  &&& NCV    & -3.8 & 3.9   & -4.0   & 4.0    & -4.0 & 4.0 & -3.9 & 3.9
  \\ \cmidrule(lr){4-12}
  & &   \multirow{4}{*}{$1$}   
  & DD-SBM & 0.0 & 0.5 & -4.0 & 4.0 & {-2.0} & {2.3} & {-2.8} & 2.9\\ 
  &&& c-SBM & -0.3 & 0.8 & -3.9 & 4.0 & -2.0 & 2.3 & -2.9 & 2.9 \\ 
  &&& CLBIC & {0.0} & {0.0} & {-3.0} & {3.0}& -2.8 & 2.8 & -2.7 & 2.7  \\ 
  &&& NCV & {0.0} & {0.0} & -4.0 & 4.0 & -3.8 & 3.9 & -2.6  & 2.7\\

  \cmidrule(lr){2-12} 
  \multirow{8}{*}{$7$}  & 
  \multirow{8}{*}{$50$}  & 
  \multirow{4}{*}{$\frac{1}{2}$}  
   &  DD-SBM & -5.5 & 5.6 & -5.9 & 5.9 & -5.9 & 5.9 & {-3.9} & 4.1 \\ 
  & && c-SBM & -5.9 & 5.9 & -5.9 & 5.9 & -6.0 & 6.0 & -5.0 & 5.1\\ 
  & && CLBIC & {-5.2} & {5.2} & {-5.3} & {5.3}& {-5.3} & {5.3 }& -5.4 & 5.5 \\ 
  &&& NCV    & -6.0 & 6.0   & -6.0   & 6.0    & -6.0 & 6.0 & -6.0 & 6.0
  \\ \cmidrule(lr){4-12}
  & &   \multirow{4}{*}{$1$}   
  & DD-SBM & -3.1 & 3.5 & -6.0 & 6.0 & {-6.0} & {6.0} & {-4.5} & 4.6\\ 
  &&& c-SBM & -4.7 & 4.8 & -5.9 & 6.0 & -5.9 & 5.9 & -4.9 & 5.0 \\ 
  &&& CLBIC & {-4.8} & {4.9} & {-5.3} & {5.3}& -5.3 & 5.3 & -4.8 & 4.9  \\ 
  &&& NCV & {-6.0} & {6.0} & -6.0 & 6.0 & -6.0 & 6.0 & -5.5  & 5.6\\

    \cmidrule(lr){3-12} 
  \multirow{8}{*}{$7$}  & 
  \multirow{8}{*}{$75$}  & 
  \multirow{4}{*}{$\frac{1}{2}$}  
   &  DD-SBM & -4.6 & 4.7 & -6.0 & 6.0 & -5.9 & 6.0 & {-4.3} & 4.4 \\ 
  & && c-SBM & -5.4 & 5.4 & -5.9 & 5.9 & -5.9 & 5.9 & -5.0 & 5.0\\ 
  & && CLBIC & {-4.8} & {4.9} & {-5.0} & {5.0}& {-5.0} & {5.0 }& -4.8 & 4.9 \\ 
  &&& NCV    & -6.0 & 6.0   & -6.0   & 6.0    & -6.0 & 6.0 & -6.0 & 6.0
  \\ \cmidrule(lr){4-12}
  & &   \multirow{4}{*}{$1$}   
  & DD-SBM & -1.4 & 2.0 & -5.9 & 6.0 & {-5.5} & {5.5} & {-4.8} & 4.8\\ 
  &&& c-SBM & -2.3 & 2.6 & -5.9 & 5.9 & -5.3 & 5.4 & -5.0 & 5.0 \\ 
  &&& CLBIC & {-3.4} & {3.5} & {-5.0} & {5.0}& -5.0 & 5.0 & -4.7 & 4.7  \\ 
  &&& NCV & {-3.2} & {3.5} & -6.0 & 6.0 & -6.0 & 6.0 & -4.8  & 4.8\\

  \bottomrule 
\end{tabular}
\caption{Bias and RMSE of $\hat{K}$.}  
\label{tab:compare-methods-all}
\end{table}
 


\end{document}